\documentclass[twoside,11pt]{article}
\usepackage{isipta}

\usepackage[utf8]{inputenc}
\inputencoding{utf8} 

\usepackage{hyperref,color,soul,booktabs}
\setulcolor{blue}
\newcommand{\BibTeX}{\textsc{B\kern-0.1emi\kern-0.017emb}\kern-0.15em\TeX}

\usepackage{url}
\usepackage{mathptmx}
\usepackage{amsmath}
\usepackage{courier}
\usepackage{enumerate}
\usepackage{enumitem,multicol}
\usepackage{tikz}
\usepackage{nicefrac}
\usepackage{bm}

\DeclareMathAlphabet\mathbfcal{OMS}{cmsy}{b}{n}

\newcommand{\nats}{\mathbb{N}}
\newcommand{\natswith}{\nats_{0}}
\newcommand{\reals}{\mathbb{R}}
\newcommand{\rationals}{\mathbb{Q}}
\newcommand{\rationalsnonneg}{\mathbb{Q}_{\geq0}}

\newcommand{\realspos}{\reals_{>0}}
\newcommand{\realsnonneg}{\reals_{\geq 0}}

\newcommand{\states}{\mathcal{X}}

\newcommand{\power}{\mathcal{P}(\hspace{-1.5pt}\states\hspace{-0.5pt})}

\newcommand{\poweron}[1]{\mathcal{P}(\hspace{-1.5pt}\states_{#1}\hspace{-0.5pt})}
\newcommand{\nonemptypower}{\mathcal{P}_{\emptyset}(\hspace{-1.5pt}\states\hspace{-0.5pt})}

\newcommand{\nonemptypoweron}[1]{\mathcal{P}_{\emptyset}(\hspace{-1.5pt}\states_{#1}\hspace{-0.5pt})}

\newcommand{\gambles}{\mathcal{G}(\hspace{-1.5pt}\states\hspace{-0.5pt})}

\newcommand{\gamblespos}{\mathcal{G}_{>0}(\hspace{-1.5pt}\states\hspace{-0.5pt})}
\newcommand{\gamblesnonneg}{\mathcal{G}_{\geq0}(\hspace{-1.5pt}\states\hspace{-0.5pt})}

\newcommand{\gambleson}[1]{\mathcal{G}(\hspace{-1.5pt}\states_{#1}\hspace{0.pt})}

\newcommand{\desir}{\mathcal{D}}

\newcommand{\ind}[1]{\mathbb{I}_{#1}}

\newcommand{\asa}{\Leftrightarrow}

\newcommand{\abs}[1]{\left\vert #1 \right\vert}

\newcommand{\coloneqq}{:\!=}

\newcommand{\lp}{\underline{P}}
\newcommand{\pr}{P}
\newcommand{\natexLP}{\underline{E}}
\newcommand{\natexUP}{\overline{E}}

\newcommand{\posi}{\mathrm{posi}}
\newcommand{\marg}{\mathrm{marg}}

\newcommand{\A}{\mathcal{A}}
\newcommand{\B}{\mathcal{B}}
\newcommand{\C}{\mathcal{C}}
\newcommand{\E}{\mathcal{E}}
\newcommand{\Y}{\mathcal{Y}}

\ShortHeadings{Independent Natural Extension for Infinite Spaces: Williams-Coherence to the Rescue}{De Bock}

\begin{document}

% Title and authors
\title{Independent Natural Extension for Infinite Spaces:\\ Williams-Coherence to the Rescue}
\author{\name Jasper De Bock \email jasper.debock@ugent.be\\
\addr Ghent University - imec, IDLab, ELIS\\
Technologiepark -- Zwijnaarde 914, 9052 Zwijnaarde, Belgium}
\maketitle
% Abstract and keywords
\begin{abstract}%   <- trailing '%' for backward compatibility of .sty file
%*** The abstract is a mandatory element that should summarize the contents of the paper. It should be confined within a single paragraph. ***
We define the independent natural extension of two local models for the general case of infinite spaces, using both sets of desirable gambles and conditional lower previsions. In contrast to~\cite{Miranda2015460}, we adopt Williams-coherence instead of Walley-coherence. We show that our notion of independent natural extension always exists---whereas theirs does not---and that it satisfies various convenient properties, including factorisation and external additivity.
\end{abstract}

\section{Introduction}\label{sec:introduction}

When probabilities are imprecise, in the sense that they are only partially specified, it is no longer clear what it means for two variables to be independent~\citep{Couso:1999wh}. One approach is to apply the standard notion of independence to every element of some set of probability measures. The alternative, called epistemic independence, is to define independence as mutual irrelevance, in the sense that receiving information about one of the variables will not effect our uncertainty model for the other. %Despite the fact that it is relatively unkown, this notion of epistemic independence has several advantages.
The advantage of this intuitive alternative is that it has a much wider scope: since epistemic independence is expressed in terms of uncertainty models instead of probabilities, it can easily be applied to a variety of such models, including non-probabilistic ones; we here consider sets of desirable gambles and conditional lower previsions.

When an assessment of epistemic independence is combined with local uncertainty models, it leads to a unique corresponding joint uncertainty model that is called the independent natural extension. If the variables involved can take only a finite number of values, this independent natural extension always exists, and it then satisfies various convenient properties that allow for the design of efficient algorithms~\citep{deCooman:2011ey,deCooman:2012vba}.
% in credal networks, which are Bayesian networks with partially specified probabilities~\cite{Cozman:2000ug}; see for example ~\citep{deCooman:2010gd}, \cite{DeBock:2014ts} and~\cite{de2015credal}. 
If the variables involved take values in an infinite set, the situation becomes more complicated. On the one hand, for the specific case of lower probabilities,~\cite{Vicig:2000vh} managed to obtain results that resemble the finite case. On the other hand, for the more general case of lower previsions,~\cite{Miranda2015460} recently found that the independent natural extension may not even exist. 

Our present contribution generalises the results of~\cite{Vicig:2000vh} to the case of conditional lower previsions, using sets of desirable gambles as an intermediate step. The key technical difference with~\cite{Miranda2015460} is that we use Williams-coherence instead of Walley-coherence. This difference turns out to be crucial because our notion of independent natural extension always exists. Furthermore, as we will see, it satisfies the same convenient properties that are known to hold in the finite case, including factorisation and external additivity.

\section{Preliminaries and Notation}\label{sec:prelim}

We use $\nats$ to denote the natural numbers without zero and let $\natswith\coloneqq\nats\cup\{0\}$. $\reals$ is the set of real numbers and $\rationals$ is the set of rational numbers. Sign restrictions are imposed with subscripts. For example, we let $\reals_{>0}$ be the set of positive real numbers and let $\rationals_{\geq0}$ be the set of non-negative rational numbers. The extended real numbers are denoted by $\overline{\reals}\coloneqq\reals\cup\{-\infty,+\infty\}$.

For any non-empty set $\states$, the power set of $\states$---the set of all subsets of $\states$---is denoted by $\power$, and we let $\nonemptypower\coloneqq\power\setminus\{\emptyset\}$ be the set of all non-empty subsets of $\states$. Elements of $\power$ are called events. A set of events $\B\subseteq\power$ is called a field if it is non-empty and closed with respect to complements and finite intersections and unions. If it is also closed with respect to countable intersections and unions, it is called a sigma field. A partition of $\states$ is a set $\B\subseteq\nonemptypower$ of pairwise disjoint non-empty subsets of $\states$ whose union is equal to $\states$. We also adopt the notational trick of identifying $\states$ with the set of atoms $\{\{x\}\colon x\in\states\}$, which allows us to regard $\states$ as a partition of $\states$.

%For any $x\in\states$, we use the notational trick of identifying $x$ with the singleton $\{x\}$. This will for example allow us to regard $\states$ as a partition of $\states$.

A bounded real-valued function on $\states$ will be called a gamble on $\states$. The set of all gambles on $\states$ is denoted by $\gambles$, the set of all non-negative gambles on $\states$ is denoted by $\mathcal{G}_{\geq0}(\states)$, and we let $\gamblespos\coloneqq\mathcal{G}_{\geq0}(\states)\setminus\{0\}$ be the set of all non-negative non-zero gambles. For any set of gambles $\A\subseteq\gambles$, we let
\vspace{-2pt}
\begin{equation}\label{eq:posi}
\posi(\A)
\coloneqq
\left\{
\sum_{i=1}^n \lambda_if_i
\colon
n\in\nats,
\lambda_i\in\realspos,f_i\in\mathcal{A}
\right\}
\end{equation}
and
\begin{equation}\label{eq:natextop}
\E(\A)
\coloneqq
\posi\left(
\A\cup\gamblespos
\right).\vspace{10pt}
\end{equation}
Indicators are a particular type of gamble. For any $A\in\power$, the corresponding indicator $\ind{A}$ of $A$ is a gamble in $\gambles$, defined for all $x\in\states$ by $\ind{A}(x)\coloneqq1$ if $x\in A$ and $\ind{A}(x)\coloneqq0$ otherwise. %If $A$ is a singleton $\{x\}$, we also let $\ind{x}\coloneqq\ind{\{x\}}$.

Finally, for any $\B\subseteq\nonemptypower$, we will also require the notion of a non-negative $\B$-measurable gamble, which we define as a uniform limit of simple $\B$-measurable gambles.

\begin{definition}\label{def:measurable:simple}Let $\B\subseteq\nonemptypower$. We call $\smash{g\in\mathcal{G}_{\geq0}(\states)}$ a simple $\B$-measurable gamble if there are $c_0\in\reals_{\geq0}$, $n\in\natswith$ and, for all $i\in\{1,\dots,n\}$, $c_i\in\reals_{\geq0}$ and $B_i\in\B$, such that $g=c_0+\sum_{i=1}^nc_i\ind{B_i}$.
\end{definition}

% *** the difference with the usual definition is that the $c_i$ are required to be positive ***

% *** non-negative constant functions are always simple $\B$-measurable functions ***

% *** explain (and/or prove) that if $\B$ is a $\sigma$-field, we obtain the usual notion ***

\begin{definition}\label{def:measurable:uniform}Let $\B\subseteq\nonemptypower$. A gamble $g\in\mathcal{G}_{\geq0}(\states)$ is $\B$-measurable if it is a uniform limit of non-negative simple $\B$-measurable gambles, in the sense that there is a sequence $\{g_n\}_{n\in\nats}$ of simple $\B$-measurable gambles in $\mathcal{G}_{\geq0}(\states)$ such that $\lim_{n\to+\infty}\sup\abs{g-g_n}=0$.
\end{definition}

Readers that are familiar with the concepts of simple and measurable functions that are common in measure theory will observe some similarities. However, there are also some important differences. On the one hand, our definitions are more restrictive: we only consider bounded non-negative functions, Definition~\ref{def:measurable:simple} requires that the coefficients $c_i$ are non-negative, and Definition~\ref{def:measurable:uniform} considers uniform limits instead of pointwise limits. On the other hand, our definitions are more general because we allow for $\B$ to be any subset of $\nonemptypower$. Nevertheless, if $\B\cup\{\emptyset\}$ is a sigma field, we have the following equivalence.

\begin{proposition}\label{prop:measurability:equivalenceforsigmafield}
Consider any $\B\subseteq\nonemptypower$ such that $\B^*\coloneqq\B\cup\{\emptyset\}$ is a sigma field. Then for any $g\in\gamblesnonneg$, $g$ is $\B^*$-measurable in the measure-theoretic sense~\cite[Definition~10.1]{Nielsen1997} if and only if it is $\B$-measurable in the sense of Definition~\ref{def:measurable:uniform}.
\end{proposition}

% \begin{definition}\label{def:measurable:pointwise}Let $\B\subseteq\nonemptypower$. We call $g\in\mathcal{G}_{\geq0}(\states)$ a $\B$-measurable gamble if it is a limit of non-negative simple $\B$-measurable functions, in the sense that there is a sequence $\{g_n\}_{n\in\nats}$ of simple $\B$-measurable functions in $\mathcal{G}_{\geq0}(\states)$ such that, for all $x\in\states$, $\lim_{n\to\infty}\vert g(x)-g_n(x)\vert=0$.
% \end{definition}

The proof of this result is based on the following sufficient condition for $\B$-measurability, which provides a convenient tool for establishing the $\B$-measurability of a given function. In particular, 
%and perhaps most importantly, 
it implies that every non-negative gamble is $\nonemptypower$-measurable.

\begin{proposition}\label{prop:measurable:sufficient:general}
Let $\B\subseteq\nonemptypower$ and $g\in\mathcal{G}_{\geq0}(\states)$. If, for all $r\in\rationalsnonneg$, the set $\{x\in\states\colon g(x)\geq r\}$ is a finite union of pairwise disjoint events in $\B\cup\{\states,\emptyset\}$, then $g$ is $\B$-measurable.
\end{proposition}

% In particular, if $\B=\nonemptypower$, this result trivially implies that every non-negative gamble is $\B$-measurable. This observation will prove to be of major importance for our present purposes. Therefore, despite its triviality, we choose to provide it with a separate statement.

\begin{corollary}\label{corol:measurable:sufficient:allsets}
Every $g\in\mathcal{G}_{\geq0}(\states)$ is $\nonemptypower$-measurable.
\end{corollary}

%*** perhaps give example of continuous symmetric function (for example Gauss-curve) on the real axis with $\B$ the set of all closed symmetric intervals around zero ***

% Explain that if $\B=\nonemptypower$, then every (!!!) function is uniformly $\B$-measurable.

%Give example to show that for $\states=\nats$, $g(x)\coloneqq\nicefrac{1}{x}$ is $\states$-measurable. (we identify the elements of $\states$ with singleton sets)

\section{Modelling Uncertainty}\label{sec:modellinguncertainty}

A subject's uncertainty about a variable $X$ that takes values $x$ in some non-empty set $\states$ can be mathematically represented in various ways. The most popular such method is perhaps probability theory, but it is by no means the only one, nor is it the most general one. In order for our results to have a broader scope, we here adopt the frameworks of sets of desirable gambles and conditional lower previsions. 

%Although these frameworks are relatively ill-known, they have the major advantage that most other uncertainty models correspond to special cases, including for example lower expectations, expectations, lower probabilities, probabilities and belief functions.% see~\cite{augustin2013:itip} for more information.

%Although these connections are important, and important to appreciate the generality of our results

The main aim of this section is to provide an overview of the basic technical aspects of these frameworks, as these will be essential to the rest of the paper. Notably, we do not impose any constraints on the cardinality of $\states$: it may be finite, countably infinite or uncountably infinite. Connections with other---perhaps better known---models for uncertainty, including probability theory, will be discussed briefly at the end.

%The central idea behind sets of desirable gambles and conditional lower previsions is to model a subject's uncertainty about $X$ by considering his attitude towards gambles---bets---on $X$. In particular, $f(x)$ is interpreted as a---possibly negative---reward in some linear utility scale, the size and sign of which depends on the unkown value $x$ of $X$.

The basic idea behind \emph{sets of desirable gambles} is to model a subject's uncertainty about $X$ by considering his attitude towards gambles---bets---on $\states$. In particular, we consider the gambles $f\in\gambles$ that he finds \emph{desirable}, in the sense that he is willing to engage in a transaction where, once the actual value $x\in\states$ of $X$ is known, he will receive a---possibly negative---reward $f(x)$ in some linear utility scale. Even more so, he prefers these desirable gambles over the status quo, that is, over not conducting any transaction at all. A set of desirable gambles is called coherent if it satisfies the following rationality requirements.

%\subsection{Sets of Desirable Gambles}

\begin{definition}%[Coherent set of desirable gambles]
\label{def:SDG}
A coherent set of desirable gambles $\desir$ on $\states$ is a subset of\/ $\gambles$ such that, for any two gambles $f,g\in\gambles$ and any non-negative real number $\lambda\in\realspos$:
\vspace{2pt}

\begin{enumerate}[label=\emph{D\arabic*:},ref=D\arabic*]
\item
if $f\geq0$ and $f\neq0$, then $f\in\desir$;\label{def:SDG:partialgain}
\item
if $f\in\desir$ then $\lambda f\in\desir$;\label{def:SDG:homo}
\item
if $f,g\in\desir$, then $f+g\in\desir$;\label{def:SDG:convex}
\item
if $f\leq0$, then $f\notin\desir$.\label{def:SDG:partialloss}
\vspace{2pt}
\end{enumerate}
\end{definition}
\noindent
Despite their simplicity, sets of desirable gambles offer a surprisingly powerful framework for modelling uncertainty; see for example~\citep{Walley:2000ef} and~\citep{Quaeghebeur:2014tjb}. %Furthermore, as we will see, they allow for an elegant treatment of the theory of conditional lower previsions. 
For our present purposes though, all we need for now is Definition~\ref{def:SDG}.

% \noindent
% QUESTION: Some authors relace \ref{def:SDG:partialloss} by ``if $f\leq0$ and $\sup\{f(x)\colon f(x)<0\}<0$, then $f\notin\desir$''. %\footnote{Since the supremum of the empty set is equal to $-\infty$, this also implies that $0\notin\desir$.} 
% In fact, Williams does so in the beginning of his paper. Why is this done? Is there some reason why this is important? And if it is indeed important: can the results further on in this paper be adapted accordingly? (this does not seem trivial, at least not for Proposition~\ref{prop:subsetproductcoherent})

% \vspace{5pt}

%\subsection{Full conditional lower previsions}

\emph{Conditional lower previsions} also model a subject's uncertainty about $X$ by considering his attitude towards gambles on $\states$. However, in this case, instead of considering sets of gambles, we consider the prices at which a subject is willing to buy these gambles. Let
\begin{equation*}
\mathcal{C}(\states)
\coloneqq
\gambles\times\nonemptypower
\end{equation*}
be the set of all pairs $(f,B)$, where $f$ is a gamble on $\states$ and $B$ is a non-empty subset of $\states$---an event. A conditional lower prevision is then defined as follows.
\begin{definition}%[Conditional lower prevision]
A conditional lower prevision $\lp$ on $\C\subseteq\C(\states)$ is a map
\begin{equation*}
\lp\colon
\C
\to
\overline{\reals}
\colon
(f,B)\to \lp(f\vert B).\vspace{7pt}
\end{equation*}
\end{definition}
For any $(f,B)$ in the domain $\C$, the lower prevision $\lp(f\vert B)$ of $f$ conditional on $B$ is interpreted as a subject's supremum price $\mu$ for buying $f$, under the condition that the transaction is called off when $B$ does not happen---if $x\notin B$. In other words, $\lp(f\vert B)$ is the supremum value of $\mu$ for which he is willing to engage in a transaction where he receives $f(x)-\mu$ if $x\in B$ and zero otherwise, and furthermore prefers this transaction to the status quo.

%which the gamble $[f-\mu]\ind{B}$ is considered desirable.

It is also possible to consider conditional upper previsions $\overline{P}(f\vert B)$, which are interpreted as infimum selling prices. However, since selling $f$ for $\mu$ is equivalent to buying $-f$ for $-\mu$, we have that $\overline{P}(f\vert B)=-\lp(-f\vert B)$. For that reason, we will mainly focus on conditional lower previsions. Unconditional lower previsions correspond to the special case where $B=\states$ for all $(f,B)\in\C$; we then use the shorthand notation $\lp(f)\coloneqq \lp(f\vert\states)$ and call $\lp(f)$ the lower prevision of $f$. Similarly, we refer to $\overline{P}(f)\coloneqq\overline{P}(f\vert\states)$ as the upper prevision of $f$.

Because of their interpretation in terms of buying prices for gambles, a particularly intuitive way to obtain a conditional lower prevision $\lp$ is to derive it from a set of gambles $\desir$. In particular, for every $\desir\subseteq\gambles$, we let 
\begin{equation}\label{eq:LPfromD}
\lp_\desir(f\vert B)
\coloneqq
\sup\{\mu\in\reals\colon
[f-\mu]\ind{B}\in\desir
\}
\text{~~for all $(f,B)\in\C(\states)$.}
\vspace{3pt}
\end{equation}
A conditional lower prevision is then called coherent if can be derived from a coherent set of desirable gambles in this way.

\begin{definition}%[Coherent conditional lower prevision]
\label{def:cohlp}
A conditional lower prevision $\lp$ on a domain $\C\subseteq\C(\states)$ is coherent if there is a coherent set of desirable gambles $\desir$ on $\states$ such that $\lp$ coincides with $\lp_\desir$ on $\C$.
\end{definition}

This definition of coherence is heavily inspired by the work of~\cite{williams1975,Williams:2007eu}. The only two minor differences are that our rationality axioms on $\desir$ are slightly different from his, and that we do not impose any structure on the domain $\C$. Nevertheless, when the domain $\C$ satisfies the structural constraints in~\citep{Williams:2007eu}, Definition~\ref{def:cohlp} is equivalent to that of Williams. More generally, as the following result establishes, it is equivalent to the structure-free notion of Williams-coherence that was developed by \cite{Pelessoni:2009co}. %We will simply refer to it as coherence, but it should

%Also equivalent to Matthias and Gert (although they consider unbounded gambles)

\begin{proposition}\label{prop:equivalentToPelessoniAndVicig}
A conditional lower prevision $\lp$ on $\C\subseteq\C(\states)$ is coherent if and only if it is real-valued and, for all $n\in\natswith$ and all choices of $\lambda_0,\dots,\lambda_n\in\reals_{\geq0}$ and $(f_0,B_0),\dots,(f_n,B_n)\in\mathcal{C}$:
\begin{equation}\label{eq:prop:equivalentToPelessoniAndVicig}
\sup_{x\in B}
\Big(\,
\sum_{i=1}^n
\lambda_i\ind{B_i}(x)
[f_i(x)-\lp(f_i\vert B_i)]
-\lambda_0\ind{B_0}(x)
[f_0(x)-\lp(f_0\vert B_0)]
\Big)
\geq0,
\vspace{6pt}
\end{equation}
where we let $B\coloneqq\cup_{i=0}^nB_i$.
\end{proposition}

The advantage of this alternative characterisation is that it is expressed directly in terms of lower previsions. Nevertheless, we consider Equation~\eqref{eq:prop:equivalentToPelessoniAndVicig} to be less intuitive than Definition~\ref{def:cohlp}, which is why we prefer the latter. 
%The main advantage of our definition is its intuitive interpretation in terms of sets of desirable gambles.

From a mathematical point of view, Definition~\ref{def:cohlp} also has the advantage that it allows for simple and elegant proofs of some well-known results. For example, it follows trivially from our definition of coherence that the domain of a coherent conditional lower prevision can be arbitrarily extended while preserving coherence, whereas deriving this result directly from Equation~\ref{eq:prop:equivalentToPelessoniAndVicig} is substantially more involved; see for example the proof of \cite[Proposition~1]{Pelessoni:2009co}. Furthermore, our definition also allows for a very natural derivation of the so-called \emph{natural extension} of $\lp$, that is, the most conservative extension of $\lp$ to $\C(\states)$. In particular, instead of having to derive this natural extension directly, Definition~\ref{def:cohlp} allows us to rephrase this problem into a closely related yet simpler question: what is the smallest coherent set of desirable gambles $\desir$ on $\states$ such that $\lp_\desir$ coincides with $\lp$ on $\C$? The answer turns out to be surprisingly simple.

\begin{proposition}\label{prop:smallestSDGfromLP}
Consider a coherent conditional lower prevision $\lp$ on $\C\subseteq\C(\states)$ and let
\begin{equation}\label{eq:AfromLP}
\A_{\lp}
\coloneqq
\big\{
[f-\mu]\ind{B}
\colon
(f,B)\in\C, \mu<\lp(f\vert B)
\big\}
~~\text{and}~~
\E(\lp)\coloneqq\E(\A_{\lp}).
\end{equation}
Then $\E(\lp)$ is a coherent set of desirable gambles on $\states$ and $\lp_{\E(\lp)}$ coincides with $\lp$ on $\C$. Furthermore, for any other coherent set of desirable gambles $\desir$ on $\states$ such that $\lp_{\desir}$ coincides with $\lp$ on $\C$, we have that $\E(\lp)\subseteq\desir$.
\end{proposition}

Abstracting away some technical details, the reason why this result holds should be intuitively clear. First, since conditional lower previsions are interpreted as called-off supremum buying prices, we see that the gambles in $\A_{\lp}$ should be desirable. Combined with~\ref{def:SDG:partialgain}--\ref{def:SDG:convex}, the desirability of the gambles in $\E(\lp)$ then follows.

%explain the link with supremum buying price and desirability! argue why these gambles should be desirable... the following result makes this formal.

%In other words: for any coherent $\lp$, among all the coherent $\desir$ such that $\lp_\desir$ coincides with $\lp$, $\E(\lp)$ is the unique smallest one.

% \begin{proposition}
% Let $\lp$ be a coherent conditional lower prevision on $\C\subseteq\C(\states)$. Then $\E(\lp)$ is a coherent set of desirable gambles on $\states$.
% \end{proposition}

%Because of this result, we can use Equation~\eqref{eq:LPfromD}

%Since $\E(\lp)$ is included in every $\desir$, and since smaller sets lead to smaller lower previsions, it follows that the natural extension is given by

Since smaller sets of desirable gambles lead to more conservative---pointwise smaller---lower previsions, we conclude that the natural extension of $\lp$ is given by
\begin{equation}\label{eq:naturalextension}
\natexLP(f\vert B)\coloneqq\lp_{\E(\lp)}(f\vert B)
\text{~~for all 
%$f\in\gambles$ and $B\in\nonemptypower$.}
$(f,B)\in\C(\states)$.}
\end{equation}
The following proposition  provides a formal statement of this result.

\begin{proposition}\label{prop:naturalextension:full}
Let $\lp$ be a coherent conditional lower prevision on $\C\subseteq\C(\states)$. Then $\natexLP$, as defined by Equation~\eqref{eq:naturalextension}, is the pointwise smallest coherent conditional lower prevision on $\C(\states)$ that coincides with $\lp$ on $\C$.
\end{proposition}

%This alternative characterisation has the advantage that it is expressed directly in terms of $\lp$, in the sense that it does not require an understanding of the concept of coherent sets of desirable gambles. Nevertheless, we prefer our definition because we consider it to be simpler and more intuitive.

All in all, we conclude that Definition~\ref{def:cohlp} provides an intuitive as well as mathematically convenient characterisation of Williams-coherence that is furthermore equivalent to the structure-free version of~\cite{Pelessoni:2009co}. 
%On the one side, our focus on sets of desirable gambles adds intuition and mathematical simplicity. 
From a technical point of view, this equivalence will not be important further on, since all of our arguments will be based on the connection with sets of desirable gambles.
%Because of this connection, we can import all kinds of well known results into our framework...
From a practical point of view though, this equivalence is highly important, because the Williams-coherent conditional lower previsions that are considered in~\citep{Pelessoni:2009co} are well-known to include as special cases a variety of other uncertainty models, including expectations, lower expectations, probabilities, lower probabilities and belief functions; lower probabilities, for example, can be obtained by restricting the domain of $\lp$ to indicators. For that reason, all of our results can be immediately applied to these special cases as well. 
A detailed treatment of these special cases, however, does not fit within the page constraints of this contribution, and therefore falls beyond the scope our present work.

\section{Epistemic Independence}\label{sec:independence}

Having introduced our main tools for modelling uncertainty, the next step towards developing a notion of independent natural extension is to agree on what we mean by independence. Within the context of lower previsions, there are basically two main options.

 The first approach, which we will not consider here, is to interpret lower previsions as lower expectations, that is, as tight lower bounds on the expectations that correspond to some set of probability measures, and to then impose the usual notion of independence on each of the probability measures in that set. This approach has the advantage of being familiar, but is restricted in scope because it can only be applied to uncertainty models that are expressed in terms of probabilities. 

 The second approach, which is the one that we will adopt here, is to regard independence as an assessment of mutual irrelevance. In particular, we say that $X_1$ and $X_2$ are independent if our uncertainty model for $X_1$ is not affected by conditioning on information about $X_2$, and vice versa. This definition can easily be applied to a probability measure, and then yields the usual notion of independence. However, and that is what makes this approach powerful and intuitive, it can just as easily be applied to lower previsions, sets of desirable gambles, or any other type of uncertainty model. This type of independence is usually referred to as epistemic independence. The aim of this section is to formalize this concept for the case of two variables, in terms of sets of desirable gambles and conditional lower previsions.

Consider two variables $X_1$ and $X_2$ where, for every $i\in\{1,2\}$, $X_i$ takes values $x_i$ in a non-empty set $\states_i$ that may be uncountably infinite, and let $X\coloneqq(X_1,X_2)$ be the corresponding joint variable that takes values $x\coloneqq(x_1,x_2)$ in $\states_1\times\states_2$. In this context, whenever convenient, we will identify $B\in\nonemptypoweron{1}$ with $B\times\states_2$ and $B\in\nonemptypoweron{2}$ with $\states_1\times B$. Similarly, for any $i\in\{1,2\}$, we will identify $f\in\gambleson{i}$ with its cylindrical extension to $\mathcal{G}(\states_1\times\states_2)$, defined by 
\begin{equation*}
f(x_1,x_2)\coloneqq f(x_i)
\text{~~for all $x=(x_1,x_2)\in\states_1\times\states_2$}.
\end{equation*}
In order to make this explicit, we will then often denote this cylindrical extension by $f(X_i)$. In this way, for example, for any $f\in\gambleson{2}$ and $B\in\poweron{1}$, we can write $f(X_2)\ind{B}(X_1)$ to denote a gamble in $\mathcal{G}(\states_1\times\states_2)$ whose value in $(x_1,x_2)$ is equal to $f(x_2)$ if $x_1\in B$ and equal to zero otherwise. Using these conventions, for any set of gambles $\desir$ on $\states_1\times\states_2$, we define the marginal models\vspace{2pt}
\begin{equation*}
\marg_1(\desir)\coloneqq\{f\in\gambleson{1}\colon f(X_1)\in\desir\}
~~\text{and}~~
\marg_2(\desir)\coloneqq\{f\in\gambleson{2}\colon f(X_2)\in\desir\}
\vspace{2pt}
\end{equation*}
% and
% \begin{equation*}
% \marg_2(\desir)\coloneqq\{f\in\gambleson{2}\colon f(X_2)\in\desir\}
% \vspace{8pt}
% \end{equation*}
and, for any events $B_1\in\nonemptypoweron{1}$ and $B_2\in\nonemptypoweron{2}$, the conditional models
\vspace{3pt}
\begin{equation*}
\marg_1(\desir\rfloor B_2)\coloneqq\{f\in\gambleson{1}\colon f(X_1)\ind{B_2}(X_2)\in\desir\}\vspace{-6pt}
\end{equation*}
and
\begin{equation*}
\marg_2(\desir\rfloor B_1)\coloneqq\{f\in\gambleson{2}\colon f(X_2)\ind{B_1}(X_1)\in\desir\}.\vspace{8pt}
\end{equation*}
%If $B_2=\{x_2\}$, with $x_2\in\states_2$, we write $\marg_1(\desir\rfloor x_2)\coloneqq\marg_1(\desir\rfloor\{x_2\})$. Similarly, if $B_1=\{x_1\}$, with $x_1\in\states_1$, we write $\marg_2(\desir\rfloor x_1)\coloneqq\marg_2(\desir\rfloor\{x_1\})$. 
Conditioning and marginalisation both preserve coherence: if $\desir$ is a coherent set of desirable gambles on $\states_1\times\states_2$, then $\marg_1(\desir)$ and $\marg_1(\desir\rfloor B_2)$ are coherent sets of desirable gambles on $\states_1$, and $\marg_2(\desir)$ and $\marg_2(\desir\rfloor B_1)$ are coherent sets of desirable gambles on $\states_2$.

That being said, let us now recall our informal definition of epistemic independence, which was that the uncertainty model for $X_1$ is not affected by conditioning on information about $X_2$, and vice versa. In the context of sets of desirable gambles, this can now be formalized as follows:
\begin{equation*}
\marg_1(\desir\vert B_2)=\marg_1(\desir)
~~\text{and}~~
\marg_2(\desir\vert B_1)=\marg_2(\desir).
\end{equation*}
The only thing that is left to specify are the conditioning events $B_1$ and $B_2$ for which we want this condition to hold. We think that the most intuitive approach is to impose this for every $B_1\in\nonemptypoweron{1}$ and $B_2\in\nonemptypoweron{2}$, and will call this epistemic subset-independence. However, this is not what is usually done. The conventional approach, which we will refer to as epistemic value-independence, is to focus on singleton events of the type $B_1=\{x_1\}$ and $B_2=\{x_2\}$; see for example~\citep{Walley:1991vk} and~\citep{deCooman:2012vba}. We believe this conventional approach to be flawed and will argue against it further on. Until then, we postpone this debate by adopting a very general approach that subsumes the former two as special cases. In particular, for every $i\in\{1,2\}$, we simply fix a generic set of conditioning events $\B_i\subseteq\nonemptypoweron{i}$. Epistemic value-independence corresponds to choosing $\B_i=\states_i$, whereas epistemic subset-independence corresponds to choosing $\B_i=\nonemptypoweron{i}$.

For sets of desirable gambles, this leads us to the following definition.

%For now, we focus on proving results that do not depend on the choice of $\B_i$. However, the choice does matter! We postpone this debate to Section ???

%Independence is essentially about that conditioning does not matter... but which conditioning events should we consider? We remain very general, by using $\B_1$ and $\B_2$.

% \noindent
% In most cases, $\B_1$ and $\B_2$ will usually be a partition or an algebra (possibly sigma-additive). Particularly interesting cases are value-irrelevance, subset-irrelevance, and perhaps irrelevance with respect to the Borel algebra.
% \vspace{10pt}

\begin{definition}%[epistemic independence for sets of desirable gambles]
\label{def:irrSDG}
Let $\desir$ be a coherent set of desirable gambles on $\states_1\times\states_2$. Then $\desir$ is epistemically independent if, for any $i$ and $j$ such that $\{i,j\}=\{1,2\}$:
\begin{equation*}
\marg_i(\desir\rfloor B_j)=\marg_i(\desir)
\text{~~for all $B_j\in\B_j$.}
\end{equation*}
% \begin{equation*}
% (\forall B_1\in\B_1)
% ~~
% \marg_2(\desir\rfloor B_1)=\marg_2(\desir).
% \end{equation*}
% Similarly, $X_2$ is said to be epistemically irrelevant to $X_1$ if
% %\vspace{-2pt}
% \begin{equation*}
% (\forall B_2\in\B_2)
% ~~
% \marg_1(\desir\rfloor B_2)=\marg_1(\desir).
% \vspace{4pt}
% \end{equation*}
% If $X_1$ and $X_2$ are epistemically irrelevant to each other, we say that they are epistemically independent.
\end{definition}
%Epistemic value-independence and subset-independence 

% \noindent
% TODO: eplain that value- and subset-irrelevance are special cases
% \vspace{10pt}

For coherent lower previsions, as a prerequisite for defining epistemic independence, we require that the domain $\C\subseteq\C(\states_1\times\states_2)$ is independent, by which we mean that for any $i$ and $j$ such that $\{i,j\}=\{1,2\}$, any pair $(f_i,B_i)\in\C(\states_i)$ and any event $B_j\in\B_j$:
\begin{equation}\label{eq:independentdomain}
(f_i,B_i)\in\C
\asa
(f_i,B_i\cap B_j)\in\C.
\end{equation}
Other than that, we impose no restrictions on $\C$; its elements $(f,B)\in\C$ are for example not restricted to the types that appear in Equation~\eqref{eq:independentdomain}. As a result, the following definition of epistemic independence is applicable beyond the context of lower previsions. For example, by restricting the domain to indicators, we obtain a notion of epistemic independence that applies to conditional lower probabilities. A detailed discussion of these special cases, however, is left as future work.
% \begin{definition}
% A domain $\C\subseteq\C(\states_1\times\states_2)$ is independent if, for any $i$ and $j$ such that $\{i,j\}=\{1,2\}$, any pair $(f_i,B_i)\in\C(\states_i)$ and any event $B_j\in\B_j$:
% \begin{equation*}
% (f_i,B_i)\in\C
% \asa
% (f_i,B_i\times B_j)\in\C.
% \end{equation*}
% \end{definition}

\begin{definition}\label{def:epistemicindependence:LP}
Let $\C\subseteq\C(\states_1\times\states_2)$ be an independent domain. A coherent conditional lower prevision $\lp$ on $\C$ is then epistemically independent if, for any $i$ and $j$ such that $\{i,j\}=\{1,2\}$:
\begin{equation*}
\lp(f_i\vert B_i)=\lp(f_i\vert B_i\cap B_j)
\text{~~for all $(f_i,B_i)\in\C$ and $B_j\in\B_j$.}
\end{equation*}
\end{definition}
Another important feature of this definition is that $B_j$ is not only irrelevant to unconditional local lower previsions of the form $\lp(f_i)$---in the sense that $\lp(f_i)=\lp(f_i\vert B_j)$---but also to conditional local lower previsions such as $\lp(f_i\vert B_i)$---in the sense that $\lp(f_i\vert B_i)=\lp(f_i\vert B_i\cap B_j)$. This type of irrelevance is called h-irrelevance; see~\cite{Cozman:2013us} and~\cite{de2015credal}. Note, however, that this feature is optional within our framework; it only appears when $\C$ is sufficiently large. If $B_i=\states_i$ for all $(f_i,B_i)\in\C$, our definition reduces to the simple requirement that $\lp(f_i)=\lp(f_i\vert B_j)$. 

% \noindent
% TODO: eplain that value- and subset-irrelevance are again special cases, and that in this case, depending on our choice of local domains, we can also do h-irrelevance versions of them.
% \vspace{10pt}

% With any independent domain $\C\subseteq\C(\states_1\times\states_2)$, we can associate two sets of conditioning events $\B_1(\C)$ and $\B_2(\C)$, defined by
% \begin{equation*}
% \B_{1}(\C)
% \coloneqq
% \big\{
% B_1\in\nonemptypoweron{1}
% \colon
% (f_{2},B_{2}\times B_1)\in\C
% \text{~for some~}
% (f_{2},B_{2})\in\C
% \big\}
% \vspace{-6pt}
% \end{equation*}
% and
% \begin{equation*}
% \B_{2}(\C)
% \coloneqq
% \big\{
% B_2\in\nonemptypoweron{2}
% \colon
% (f_1,B_1\times B_2)\in\C
% \text{~for some~}
% (f_1,B_1)\in\C
% \big\}
% \end{equation*}

% *** factorisation and things like that will most likely require that $B_i(\C)=\nonemptypoweron{i}$ ***

% *** extension will require that $B_i(\C)=B_i(\C')$ ***

% *** value-independent product corresponds to choosing $B_i(\C)=\states_i$ ***

\section{The Independent Natural Extension}\label{sec:indnatext}

%\subsection{In terms of sets of desirable gambles}

All of that being said, we are now finally ready to introduce our central object of interest, which is the \emph{independent natural extension}. Basically, the question to which this concept provides an answer is always the same: given two local uncertainty models and an assessment of epistemic independence, what then should be the corresponding joint model? The answer, however, depends on the specific framework that is being considered.

%The reason why this question is non-trivial, is because an epistemically independent uncertainty model is not uniquely determined by its marginal models.

Within the framework of sets of desirable gambles, the local uncertainty models are coherent sets of desirable gambles. In particular, for each $i\in\{1,2\}$, we are given a coherent set of desirable gambles $\desir_i$ on $\states_i$. The aim is to combine these local models with an assessment of epistemic independence to obtain a coherent set of desirable gambles $\desir$ on $\states_1\times\states_2$. The first requirement on $\desir$, therefore, is that it should have $\desir_1$ and $\desir_2$ as its marginals, in the sense that $\marg_i(\desir)=\desir_i$ for all $i\in\{1,2\}$. The second is that $\desir$ should be epistemically independent. If both requirements are met, $\desir$ is called an independent product of $\desir_1$ and $\desir_2$. The most conservative among these independent products is called the independent natural extension.

%Throughout this section, let $\desir_1$ be a coherent set of desirable gambles on $\states_1$ and let $\desir_2$ be a coherent set of desirable gambles on $\states_2$. 
\begin{definition}\label{def:subsetindependentproduct}
An independent product of $\desir_1$ and $\desir_2$ is an epistemically independent coherent set of desirable gambles $\desir$ on $\states_1\times\states_2$ that has $\desir_1$ and $\desir_2$ as its marginals.
%---in the sense that $\marg_1(\desir)=\desir_1$ and $\marg_2(\desir)=\desir_2$---and according to which $X_1$ and $X_2$ are epistemically independent.
\end{definition}

\begin{definition}\label{def:subsetindependentnatex}
The independent natural extension of $\desir_1$ and $\desir_2$ is the smallest independent product of $\desir_1$ and $\desir_2$.
\end{definition}

If all we know is that $\desir$ is epistemically independent and has $\desir_1$ and $\desir_2$ as its marginal models, then the safest choice for $\desir$---the only choice that does not require any additional assessments---is their independent natural extension, provided of course that it exists. In order to show that it always does, we let
\begin{equation}\label{eq:indnatext:SDG}
\desir_1\otimes\desir_2
\coloneqq
\E
\left(
\A_{1\to2}
\cup
\A_{2\to1}
\right),
\vspace{-4pt}
\end{equation}
with
\begin{equation}\label{eq:A12s}
\A_{1\to2}
\coloneqq
\left\{
f_2(X_2)\ind{B_1}(X_1)
\colon
f_2\in\desir_2, 
B_1\in\B_1\cup\{\states_1\}
\right\}
\vspace{-3pt}
\end{equation}
and
\begin{equation}\label{eq:A21s}
\A_{2\to1}
\coloneqq
\left\{
f_1(X_1)\ind{B_2}(X_2)
\colon
f_1\in\desir_1, 
B_2\in\B_2\cup\{\states_2\}
\right\}.
\vspace{8pt}
\end{equation}
The following result establishes that $\desir_1\otimes\desir_2$ is the independent natural extension of $\desir_1$ and $\desir_2$.

%We start by establishing the coherence of this set.

%Next, we show that $\desir_1\otimes\desir_2$ is also an independent product.

%By combining Propositions~\ref{prop:productcoherent:SDG} and~\ref{prop:productindependent:SDG}, it is now but a small step to establish the following main result of this section.

\begin{theorem}\label{theo:natext:SDG}
$\desir_1\otimes\desir_2$ is the independent natural extension of $\desir_1$ and $\desir_2$.
\end{theorem}

%\subsection{In terms of conditional lower previsions}

Similar concepts can be defined for conditional lower previsions as well. In that case, the local uncertainty models are coherent conditional lower previsions. In particular, for every $i\in\{1,2\}$, we are given a coherent conditional lower prevision $\lp_i$ on some freely chosen local domain $\C_i\subseteq\C(\states_i)$. %As a result of this freedom, $\lp_i$ can also be an unconditional lower prevision, a lower probability, a probability measure, etcetera. 
The aim is now to construct an epistemically independent coherent conditional lower prevision $\lp$ on $\C\subseteq\C(\states_1\times\states_2)$ that has $\lp_1$ and $\lp_2$ as its marginals, in the sense that $\lp$ coincides with $\lp_1$ and $\lp_2$ on their local domain: $\lp(f_i\vert B_i)=\lp_i(f_i\vert B_i)$ for all $i\in\{1,2\}$ and $(f_i,B_i)\in\C_i$. As before, a model that meets these criteria is then called an independent product, and the most conservative among them is called the independent natural extension. Clearly, in order for these notions to make sense, the global domain $\C$ must at least include the local domains $\C_1$ and $\C_2$ and must furthermore be independent in the sense of Equation~\eqref{eq:independentdomain}. The definitions and results below take this for granted.

% This includes as particular cases unconditional lower previsions, lower probabilities, sigma-additive probabilities, ...
% \vspace{10pt}

%Let $\natexLP_1$ be the natural extension of $\lp_1$ to $\C(\states_1)$ and let $\natexLP_2$ be the natural extension of $\lp_2$ to $\C(\states_2)$.

% Let $\C\subseteq\C(\states_1\times\states_2)$ be an independent domain that includes $\C_1$ and $\C_2$.
% \vspace{10pt}

\begin{definition}\label{def:independentproduct:LP}
 An independent product of $\lp_1$ and $\lp_2$ is an epistemically independent coherent conditional lower prevision on $\C$ that has $\lp_1$ and $\lp_2$ as its marginals.
\end{definition}

\begin{definition}\label{def:independentnatex:LP}
The independent natural extension of $\lp_1$ and $\lp_2$ is the point-wise smallest independent product of $\lp_1$ and $\lp_2$.
\end{definition}

Here too, if all we know is that $\lp$ is epistemically independent and has $\lp_1$ and $\lp_2$ as its marginal models, then the safest choice for $\lp$---the only choice that does not require any additional assessments---is the independent natural extension, provided that it exists. The following result establishes that it does, by showing that it is a restriction of the operator $\lp_1\otimes\lp_2$, defined by
\vspace{2pt}
\begin{equation}\label{eq:indnatex:LP}
(\lp_1\otimes\lp_2)(f\vert B)
\coloneqq
\lp_{\desir}(f\vert B)
%\sup\big\{\mu\in\reals\colon
%[f-\mu]\ind{B}\in\E(\lp_1)\otimes\E(\lp_2)
%\big\}
\text{~~for all $(f,B)\in\C(\states_1\times\states_2)$, with $\desir=\E(\lp_1)\otimes\E(\lp_2)$.}
\vspace{2pt}
\end{equation}

\begin{theorem}\label{theo:natext:LP}
The independent natural extension of $\lp_1$ and $\lp_2$ is the restriction of $\lp_1\otimes\lp_2$ to $\C$.
\end{theorem}
Interestingly, as can be seen from this result, the choice of the joint domain $\C$ does not affect the resulting independent natural extension, in the sense that any $\C$ that includes $(f,B)$ will lead to the same value of $(\lp_1\otimes\lp_2)(f\vert B)$. For that reason, we will henceforth assume without loss of generality that $\C=\C(\states_1\times\states_2)$.
%\vspace{10pt}

% \noindent
% TODO: show that $\lp_1\otimes\lp_2$ remains informative and that---despite the fact that it does not impose conglomerability---it does not suffer from the issues that are raised in Reference~\cite{Miranda2015460}.
% \vspace{10pt}

% \noindent
% TODO: explain how this implies that, provided that the local assessments are unconditional, the use of h-irrelevance makes no difference on the unconditional joint? (so, essentially, we get h-irrelevance for free, even if we do not impose it)
% \vspace{10pt}

%TODO: explain the influence of enlarging the local domains? (If I add a result about this, I can use it to simplify a step in the proof of Lemma~\ref{lemma:fact-add-simple-geq} in the appendix)

\section{On the Choice of Conditioning Events}\label{sec:choiceofevents}

%The main issue that we have so far left untouched, is how to choose the sets of conditioning events $\B_1$ and $\B_2$. 

The fact that the existence results in the previous section are valid regardless of the choice of $\B_1$ and $\B_2$ should not be taken to mean that this choice does not affect the model. In some cases, it most definitely does. In the remainder of this contribution, we will study the extend to which it does, and how it affects the properties of the resulting notion of independent natural extension.

As a first observation, we note that larger sets of conditioning events correspond to stronger assessments of epistemic independence, and therefore lead to more informative joint models. For example, as can be seen from Equations~\eqref{eq:indnatext:SDG}--\eqref{eq:A21s}, adding events to $\B_1$ and $\B_2$ leads to a larger---more informative---set of desirable gambles $\desir_1\otimes\desir_2$. Similarly, as can be seen from Equation~\eqref{eq:indnatex:LP}, it leads to a joint lower prevision that is higher---and therefore again more informative.
There is one important exception to this observation though, which occurs when we add conditioning events that are a finite disjoint union of other conditioning events. In that case, the resulting notion of independent natural extension does not change.

\begin{proposition}\label{prop:addfiniteunionstoB}
For each $i\in\{1,2\}$, let $\B'_i$ be a superset of $\B_i$ that consists of finite disjoint unions of events in $\B_i$. Replacing $\B_1$ by $\B'_1$ and $\B_2$ by $\B'_2$ then has no effect on the resulting independent natural extension $\desir_1\otimes\desir_2$ or $\lp_1\otimes\lp_2$. 
\end{proposition}

As a particular case of this result, it follows that if $\B_i$ is a finite partition of $\states_i$, we can replace it by the generated algebra---minus the empty event. As an even more particular case, if $\states_1$ and $\states_2$ are finite, we find that epistemic value- and subset-independence lead to the same notion of independent natural extension. For that reason, in the finite case, it does not really matter which of these two types of epistemic independence is adopted.

In the infinite case though, the difference does matter, and the debate between epistemic value- and subset-independence remains open. For lower previsions,~\cite{Miranda2015460} recently adopted epistemic value-independence in combination with Walley-coherence. Unfortunately, they found that the corresponding notion of independent natural extension does not always exist. They also considered the combination of epistemic value-independence with Williams-coherence, and argued that the resulting model was too weak.
For the case of lower probabilities,~\cite{Vicig:2000vh} adopted epistemic subset-independence in combination with Williams-coherence, showed that the corresponding independent natural extension always exists, and proved that it satisfies factorisation properties. Our results so far can be regarded as a generalisation of the existence results of~\cite{Vicig:2000vh}. % to the frameworks of sets of desirable gambles and conditional lower previsions. 
As we are about to show, his factorisation results can be generalised as well.

% when $\B$ is a (possibly infinite) partition (we can then replace it with the generated lattice (***?*** commutative semigroup?)Perhaps we can let $\B^\cup$ be the closure with respect to union?), or when $\B\cup\emptyset$ is closed under intersections (we can then replace it with the generated lattice, but perhaps that is not interesting)
%\vspace{10pt}

% \noindent
% As explained in the section on independence, in most cases, $\B$ will usually be a partition or an algebra (possibly sigma-additive). Particularly interesting cases are value-irrelevance (define the corresponding indnatext), subset-irrelevance (define the corresponding indnatext, and simply call it the indnatext), and perhaps irrelevance with respect to the Borel algebra.
% \vspace{10pt}

\section{Factorisation and External Additivity}\label{sec:factadd}

When $\states_1$ and $\states_2$ are finite, the independent natural extension of two lower previsions $\lp_1$ and $\lp_2$ is well-known to satisfy the properties of factorisation and external additivity~\citep{deCooman:2011ey}. Factorisation, on the one hand, states that
\begin{equation}\label{eq:finitefactorisation}
(\lp_1\otimes\lp_2)(gh)=\lp_1(g\lp_2(h))=
\begin{cases}
\lp_1(g)\lp_2(h)
&\text{ if $\lp_2(h)\geq0$}\\
\overline{P}_1(g)\lp_2(h)
&\text{ if $\lp_2(h)\leq0$},
\end{cases}
\end{equation}
where $g$ is a non-negative gamble on $\states_1$, $h$ is a gamble on $\states_2$ and $\overline{P}_1(g)\coloneqq-\lp_1(-g)$. By symmetry, the role of $1$ and $2$ can of course be reversed. External additivity, on the other hand, states that%\vspace{-4pt}
\begin{equation}\label{eq:finiteexternaladditivity}
(\lp_1\otimes\lp_2)(f+h)=\lp_1(f)+\lp_2(h)
%\vspace{4pt}
\end{equation}
where $f$ and $h$ are gambles on $\states_1$ and $\states_2$, respectively. 

Compared to the properties that are satisfied by the joint expectation of a product measure of two precise probability measures, these notions of factorisation and external additivity are rather weak. For example, for a precise product measure, additivity is not `external', in the sense that $f$ and $h$ do not have to be defined on separate variables, nor does factorisation require $g$ to be non-negative. Nevertheless, even in this weaker form, these properties remain of crucial practical importance. For example, in the context of credal networks---Bayesian networks whose local models are imprecise---they turned out to be the key to the development of efficient inference algorithms; see for example~\cite{deCooman:2010gd},~\cite{DeBock:2014ts} and~\cite{de2015credal}. Any notion of independent natural extension that aims to extend such algorithms to infinite spaces, therefore, should preserve some suitable version of Equations~\eqref{eq:finitefactorisation} and~\eqref{eq:finiteexternaladditivity}.

%In the finite case, our notion coincides with the usual one and then satisifies the properties above.

The aim of this section is to study the extent to which these equations are satisfied by the notion of independent natural extension that was developed in this paper. As we will see, the answer %depends on the choice of $\B_1$ and $\B_2$, and 
ends up being surprisingly positive. 

For all $i\in\{1,2\}$, let $\lp_i$ be a coherent conditional lower prevision on $\C_i\subseteq\C(\states_i)$, let $\natexLP_i$ be its natural extension to $\C(\states_i)$, and let $\B_i$ be a subset of $\nonemptypoweron{i}$. The independent natural extension of $\lp_1$ and $\lp_2$ then satisfies the following three properties, the first of which implies the other two as special cases.

\begin{theorem}\label{theo:fact-add-measurable}
Let $\{i,j\}=\{1,2\}$. For any $f\in\gambleson{i}$, $h\in\gambleson{j}$ and $\B_i$-measurable $g\in\mathcal{G}_{\geq0}(\states_i)$, we then have that
%\vspace{-1pt}
\begin{equation*}
(\lp_1\otimes\lp_2)(f+gh)
=\natexLP_i\big(f+g\natexLP_j(h)\big).
\vspace{6pt}
\end{equation*}
\end{theorem}

\begin{corollary}[Factorisation]\label{corol:fact-measurable}
Let $\{i,j\}=\{1,2\}$. For any $h\in\gambleson{j}$ and any $g\in\mathcal{G}_{\geq0}(\states_i)$ that is $\B_i$-measurable, we then have that
\vspace{-5pt}
\begin{equation*}
(\lp_1\otimes\lp_2)(gh)
=\natexLP_i\big(g\natexLP_j(h)\big)
=
\begin{cases}
\natexLP_i(g)\natexLP_j(h)
&\text{ if $\natexLP_j(h)\geq0$;}\\
\natexUP_i(g)\natexLP_j(h)
&\text{ if $\natexLP_j(h)\leq0$.}
\end{cases}
\vspace{6pt}
\end{equation*}
\end{corollary}

\begin{corollary}[External additivity]\label{corol:add}
For any $f\in\gambleson{1}$ and $h\in\gambleson{2}$, we have that
\begin{equation*}
{(\lp_1\otimes\lp_2)(f+h)
=\natexLP_1(f)+\natexLP_2(h)}.
\vspace{2pt}
\end{equation*}
\end{corollary}
In each of these results, if the local domains $\C_1$ and $\C_2$ are sufficiently large---that is, if they include the gambles that appear in the statement of the results---it follows from Proposition~\ref{prop:naturalextension:full} that $\natexLP_i$ and $\natexLP_j$ can be replaced by $\lp_i$ and $\lp_j$, respectively, and similarly for $\natexUP_i$ and $\overline{P}_i$. %In the arguments below, we assume that this replacement has taken place. 

That being said, let us now go back to the question of whether or not Equations~\eqref{eq:finitefactorisation} and~\eqref{eq:finiteexternaladditivity} can be generalised  to the case of infinite spaces. For the case of external additivity, it clearly follows from Corollary~\ref{corol:add} that the answer is fully positive.
Furthermore, this conclusion holds regardless of our choice for $\B_1$ and $\B_2$; they can even be empty. For factorisation, the answer does depend on $\B_1$ and $\B_2$. If we adopt epistemic subset-independence---that is, if we choose $\B_1=\nonemptypoweron{1}$ and $\B_2=\nonemptypoweron{2}$---it follows from Corollaries~\ref{corol:measurable:sufficient:allsets} and~\ref{corol:fact-measurable} that the answer is again fully positive, because $\nonemptypoweron{i}$-measurability then holds trivially. If $\B_1\cup\{\emptyset\}$ and $\B_2\cup\{\emptyset\}$ are sigma fields, the answer remains fairly positive as well, because Proposition~\ref{prop:measurability:equivalenceforsigmafield} then implies that it suffices for $g$ to be measurable in the usual, measure-theoretic sense. If we adopt epistemic value-independence---that is, if we choose $\B_1=\states_1$ and $\B_2=\states_2$---it is necessary for $g$ to be $\states_i$-measurable, which is a rather strong requirement that easily fails. For that reason, we think that for the case of infinite spaces, when it comes to choosing between epistemic value- and subset-independence, the latter should be preferred over the former. %Nevertheless, even for epistemic value-independence, the following example illustrates that factorisaton can be usefully applied to meaningful functions.

\section{Conclusions and Future Work}\label{sec:conclusions}

The main conclusion of this work is that by combining Williams-coherence with epistemic subset-independence, we obtain a notion of independent natural extension that always exists, and that furthermore satisfies factorisation and external additivity. For weaker types of epistemic independence, including epistemic value-irrelevance, the existence result and the external additivity property remain valid, but factorisation then requires measurability conditions.

We foresee several lines of future research. The first, which we expect to be rather straightforward, is to extend our results from the case of two variables to that of any finite number of variables. Next, these extended versions of our results could then be used to develop efficient algorithms for credal networks whose variables take values in infinite spaces, by suitably adapting existing algorithms for the finite case.  %~\citep{deCooman:2010gd,DeBock:2014ts,de2015credal}. 
On the more technical side, it would be useful to see whether our results can be extended to the case of unbounded functions. %; factorisation could prove to be tricky here, because our proof for Corollary~\ref{corol:fact-measurable} relies rather heavily on the fact that gambles are taken to be bounded.  
Finally, for variables that take values in Euclidean space, $\B_1$ and $\B_2$ could be restricted to the Lebesgue measurable events. Combined with an assessment of continuity, we think that this could lead to the development of a notion of independent natural extension that includes sigma additive product measures as a special case.

% \noindent
% FUTURE WORK: Generalize to more than two variables.
% \vspace{10pt}

% \noindent
% FUTURE WORK: Generalize to unbounded gambles!!! (factorisation becomes more tricky then...)
% \vspace{10pt}

% \noindent
% TODO: sigma-algebra gebruiken? Add sigma-additivity or continuity as an extra assessment?
%\vspace{10pt}

% FUTURE WORK: Build a non-conglomerable theory of model building and clearly distinguish it from updating (where conglomerability does come into play).
% \vspace{10pt}

% \noindent 
% FUTURE WORK: study the special case where the local models are precise. Are the independent natural extensions precise? If not---which I expect---then for which functions are they precise?
% \vspace{10pt}

% \noindent
% FUTURE WORK: Investigate which functions can be computed using some type of integral (for exmample Lebesgue in the precise case, or Choquet in more general cases)
% \vspace{10pt}

% \noindent
% TODO: Study whether---and if yes prove that---the results of Vicig correspond to a special case. I expect this will be the case.

%MAAR!!! Voor algoritmes (zoals in credale netwerken) gebeurt factorisatie heel vaak voor indicatoren (of `simple' functies) waardoor de hoop reel is dat dit concreet toepasbaar is, zelfs als we uitbreiden naar onbegrensde functies

\acks{I am a Postdoctoral Fellow of the Research Foundation - Flanders (FWO) and wish to acknowledge its financial support.
The research that lead to this paper was conducted during a research visit---funded by an FWO travel grant---to the Imprecise Probability Group of IDSIA (Institute Dalle Molle for Artificial Intelligence), the members of which I would like to thank for their warm hospitality.
Finally, I would also like to thank two anonymous reviewers, for their generous constructive comments, and Enrique Miranda, for commenting on a preliminary version of this paper and for suggesting the idea of adopting a general notion of epistemic independence where $\B_1$ and $\B_2$ are allowed to be arbitrary.
}

%The authors would also like to thank Alexander Erreygers, for some stimulating discussions about the computational aspects of imprecise continuous-time Markov chains.
%, and two anonymous reviewers, for their generous constructive comments that led to several improvements of this paper.}

%\vskip 0.2in
%\bibliography{general}

\appendix
%\newpage

\section{Proofs and Additional Material}

\subsection{Proofs and Additional Material for Section~\ref{sec:prelim}}
\vspace{5pt}

\begin{lemma}\label{lemma:nestedpropsofposandE}
Let $\A_1$ and $\A_2$ be two subsets of $\gambles$ such that $\A_1\subseteq\A_2$. Then
\vspace{4pt}
\begin{equation*}
\posi(\A_1)\subseteq\posi(\A_2)
~\text{ and }~
\E(\A_1)\subseteq\E(\A_2).
\end{equation*}
\end{lemma}
\begin{proof}{\bf of Lemma~\ref{lemma:nestedpropsofposandE}~}
This follows trivially from Equations~\eqref{eq:posi} and~\eqref{eq:natextop}.
\end{proof}
\vspace{-6pt}

\begin{proof}{\bf of Proposition~\ref{prop:measurability:equivalenceforsigmafield}~}
Consider any $\B\subseteq\nonemptypower$ such that $\B^*\coloneqq\B\cup\{\emptyset\}$ is a sigma field and fix some $g\in\gamblesnonneg$.

We first prove the `only if' part of the statement. So assume that $g$ is $\B^*$-measurable in the measure-theoretic sense~\cite[Definition~10.1]{Nielsen1997}. It then follows from~\cite[Corollary~10.5]{Nielsen1997} that $\{x\in\states\colon g(x)\geq r\}\in\B^*=\B\cup\{\emptyset\}$ for all $r\in\rationals_{\geq0}$. Therefore, it follows from Proposition~\ref{prop:measurable:sufficient:general} that $g$ is $\B$-measurable in the sense of Definition~\ref{def:measurable:uniform}.

We end by proving the `if' part of the statement. So assume that $g$ is $\B$-measurable in the sense of Definition~\ref{def:measurable:uniform}. This means that there is a sequence $\{g_n\}_{n\in\nats}$ of simple $\B$-measurable gambles in $\mathcal{G}_{\geq0}(\states)$ such that $\lim_{n\to+\infty}\sup\abs{g-g_n}=0$. Then on the one hand, since $\lim_{n\to+\infty}\sup\abs{g-g_n}=0$ implies that $\lim_{n\to+\infty}\abs{g(x)-g_n(x)}=0$ for all $x\in\states$, we know that $\{g_n\}_{n\in\nats}$ converges pointwise to $g$ on $\states$. On the other hand, for any $n\in\nats$, we know from Definition~\ref{def:measurable:simple} that there are $c_0\in\reals_{\geq0}$, $m\in\natswith$ and, for all $i\in\{1,\dots,m\}$, $c_i\in\reals_{\geq0}$ and $B_i\in\B$, such that $g=c_0+\sum_{i=1}^mc_i\ind{B_i}$. Let $B_0=\states$. Since $\ind{\states}=1$, and because $\B^*$ is a sigma field and therefore includes $\states$, we then find that $g=\sum_{i=0}^mc_i\ind{B_i}$, where, for all $i\in\{0,\dots,n\}$, $B_i\in\B^*$. \cite[Example~10.2]{Nielsen1997} therefore implies that $g_n$ is a $\B^*$-measurable function in the measure-theoretic sense. Since this is true for every $n\in\nats$, and because $\{g_n\}_{n\in\nats}$ converges pointwise to $g$ on $\states$, it now follows from~\cite[Corollary~10.11(a)]{Nielsen1997} that $g$ is $\B^*$-measurable in the measure-theoretic sense.
\end{proof}
\vspace{-6pt}

\begin{proof}{\bf of Proposition~\ref{prop:measurable:sufficient:general}~}
Since $g\geq0$ is a gamble and therefore by definition bounded, there is some $\alpha\in\rationals_{>0}$ such that $0\leq g<\alpha$. Fix any $n\in\nats$ and let $g_n\in\gambles$ be defined by
\begin{equation*}
g_n\coloneqq\frac{1}{n}\alpha\sum_{k=1}^{n-1}\ind{A_k},
~\text{where, for all $k\in\{1,\dots,n-1\}$, }A_k\coloneqq\Big\{x\in\states\colon g(x)\geq\frac{k}{n}\alpha\Big\}.
\vspace{2pt}
\end{equation*}
For all $x\in\states$, we then find that
\vspace{3pt}
\begin{equation*}
g_n(x)
=
\frac{k_x}{n}\alpha
\leq
g(x)
\leq
\frac{k_x+1}{n}\alpha,
\text{~where we let~}
k_x\coloneqq\max\{k\in\{0,\dots,n-1\}\colon g(x)\geq\frac{k}{n}\alpha\},
\vspace{2pt}
\end{equation*}
which implies that $\abs{g(x)-g_n(x)}\leq\nicefrac{\alpha}{n}$. Since this is true for every $x\in\states$, this allows us to infer that $\sup\abs{g-g_n}\leq\nicefrac{\alpha}{n}$.

Consider now any $k\in\{1,\dots,n-1\}$. Since $\nicefrac{k}{n}\alpha\in\rationals_{\geq0}$, it follows from our assumptions on $g$ that $A_k$ is a finite union of pairwise disjoint events in $\B\cup\{\states,\emptyset\}$. Therefore, there is some $m_k\in\nats$ and, for all $i\in\{1,\dots,m_k\}$, some $B_{k,i}\in\B\cup\{\states,\emptyset\}$ such that $\ind{A_k}=\sum_{i=1}^{m_k}\ind{B_{k,i}}$. Since this is true for every $k\in\{1,\dots,n-1\}$, it follows that $g_n=\nicefrac{\alpha}{n}\sum_{k=1}^{n-1}\sum_{i=1}^{m_k}\ind{B_{k,i}}$. Since $g_n$ is clearly non-negative, and because $\ind{\states}=1$ and $\ind{\emptyset}=0$, it now follows from Definition~\ref{def:measurable:simple} that $g_n\in\gamblesnonneg$ is a simple $\B$-measurable gamble.

So, in summary then, for any fixed $n\in\nats$, we know that we can construct a simple $\B$-measurable gamble $g_n\in\gamblesnonneg$ such that $\sup\abs{g-g_n}\leq\nicefrac{\alpha}{n}$. Definition~\ref{def:measurable:uniform} therefore clearly implies that $g$ is $\B$-measurable.
\end{proof}
\vspace{-6pt}

\begin{proof}{\bf of Corollary~\ref{corol:measurable:sufficient:allsets}~}
Immediate consequence of Proposition~\ref{prop:measurable:sufficient:general}.
\end{proof}

\vspace{-16pt}

\subsection{Proofs and Additional Material for Section~\ref{sec:modellinguncertainty}}\label{app:modellinguncertainty}
\vspace{5pt}

%*** This appendix will not be part of the published version. ***

Contrary to what the length of this section of the appendix might suggest, it should be noted that many of the results in this section are essentially well-known.
Historically, most of them date back to~\cite{williams1975,Williams:2007eu}. Our versions are basically just minor variations of his results, expressed in terms of lower previsions---instead of upper previsions---and without imposing structural constraints on the domain. Similar results can also be found in~\citep{Pelessoni:2009co}, although often without proof.

%and wer recently extended to the case of unbounded gambles as well~\cite{troffaes2013:lp}. ***
%*** proof below is based on the ideas in the precise version of the proof in~\cite{Williams:2007eu}) ***

%*** (claimed in~\cite{Vicig:2007gs} that this can be easily derived from the results in~\cite{Williams:2007eu}, however, I don't see it; claimed in~\cite{Pelessoni:2009co} that it is proved in~\cite{Williams:2007eu}) *** For unbounded gambles, these results can be found in the book of Matthias en Gert!!! So... essentially, all of this stuff should be regarded as known! Our proofs are provided for the sake of completeness and self-containedness.

\begin{lemma}\label{lemma:coherenceiffLP4}
For any $\A\subseteq\gambles$, $\E(\A)$ is a coherent set of desirable gambles on $\states$ if and only if it satisfies~\ref{def:SDG:partialloss}.
\end{lemma}
\begin{proof}{\bf of Lemma~\ref{lemma:coherenceiffLP4}~}
Since Equation~\eqref{eq:natextop} implies that $\E(\A)$ satisfies~\ref{def:SDG:partialgain},~\ref{def:SDG:homo} and~\ref{def:SDG:convex}, this follows trivially from Definition~\ref{def:SDG}.
\end{proof}
\vspace{-16pt}

\begin{lemma}\label{lemma:natextDisD}
Let $\desir$ be a coherent set of desirable gambles on $\states$. Then $\E(\desir)=\desir$.
\end{lemma}
\begin{proof}{\bf of Lemma~\ref{lemma:natextDisD}~}
$\desir$ is trivially a subset of $\E(\desir)$. The converse inclusion, that is, $\E(\desir)\subseteq\desir$, is a straightforward consequence of the coherence of $\desir$.
\end{proof}
\vspace{-16pt}

\begin{lemma}\label{lemma:extendD}
Let $\desir$ be a coherent set of desirable gambles on $\states$. If $f\in\gambles$ and $f\notin\desir\cup\{0\}$, then $\E(\desir\cup\{-f\})$ is a coherent set of desirable gambles on $\states$.
\end{lemma}
\begin{proof}{\bf of Lemma~\ref{lemma:extendD}~}
Consider any $f\in\gambles$ such that $f\notin\desir\cup\{0\}$. Because of Lemma~\ref{lemma:coherenceiffLP4}, it suffices to prove that $\E(\desir\cup\{-f\})$ satisfies \ref{def:SDG:partialloss}. So consider any $g\in\gambles$ such that $g\leq0$. In the remainder of this proof, we show that $g\notin\E(\desir\cup\{-f\})$. 

Assume \emph{ex absurdo} that $g\in\E(\desir\cup\{-f\})$. Since $\desir$ is coherent, this implies that $g=\lambda h-\mu f$, with $h\in\desir$, $\lambda,\mu\in\realsnonneg$ and $\lambda+\mu>0$. If $\mu=0$, then because $h\in\desir$, the coherence of $\desir$ implies that $g=\lambda h\in\desir$, which implies that $\desir$ does not satisfy \ref{def:SDG:partialloss}, a contradiction. Hence, it follows that $\mu>0$, which implies that $f=\nicefrac{1}{\mu}(\lambda h-g)$. Therefore, since $h\in\desir$ and $-g\geq0$, it follows from the coherence of $\desir$ that $f=0$ (if $\lambda=0$ and $g=0$) or $f\in\desir$. In both cases, we contradict our assumptions.
\end{proof}
\vspace{-6pt}

\begin{proof}{\bf of Proposition~\ref{prop:equivalentToPelessoniAndVicig}~} Consider any conditional lower prevision $\lp$ on $\C\subseteq\C(\states)$. 

We start by proving the `only if' part of the statement. So let us assume that $\lp$ is coherent. According to Definition~\ref{def:cohlp}, this implies that there is a coherent set of desirable gambles $\desir$ on $\states$ such that $\lp_\desir$ coincides with $\lp$ on $\C$. We need prove that $\lp$ is real-valued and that it satisfies Equation~\eqref{eq:prop:equivalentToPelessoniAndVicig}.

We begin by establishing that $\lp$ is real-valued. So fix any $(f,B)\in\C$. For all $\mu\in\reals$ such that $\mu<\inf f$, it then follows from the coherence of $\desir$---and~\ref{def:SDG:partialgain} in particular---that $[f-\mu]\ind{B}\in\desir$.  Similarly, for all $\mu\in\reals$ such that $\mu>\sup f$, it follows from~\ref{def:SDG:partialloss} that $[f-\mu]\ind{B}\notin\desir$. Hence, we find that $\inf f\leq\lp_\desir(f\vert B)\leq\sup f$. Since $f$ is a gamble and therefore by definition bounded, this implies that $\lp_\desir(f\vert B)$ is real-valued, which in turn implies that $\lp(f\vert B)$ is real-valued because $\lp_\desir$ coincides with $\lp$ on $\C$. Since $(f,B)\in\C$ was arbitrary, this means that $\lp$ is real-valued. 

Next, we show that $\lp$ satisfies Equation~\eqref{eq:prop:equivalentToPelessoniAndVicig}. Fix any $n\in\natswith$, choose any $\lambda_0,\dots,\lambda_n\in\reals_{\geq0}$ and $(f_0,B_0),\dots,(f_n,B_n)\in\mathcal{C}$, let $B\coloneqq\cup_{i=0}^nB_i$ and let $h\in\gambles$ be defined by
\begin{equation*}
h(x)
\coloneqq\sum_{i=1}^n
\lambda_i\ind{B_i}(x)
[f_i(x)-\lp(f_i\vert B_i)]
-\lambda_0\ind{B_0}(x)
[f_0(x)-\lp(f_0\vert B_0)]
~~\text{for all $x\in\states$}.
\end{equation*}
We need to prove that $\sup_{x\in B}h(x)\geq0$. In order to do that, we start by fixing some $\epsilon>0$. Let $\epsilon_0\coloneqq\epsilon$. Since $\lp_\desir$ coincides with $\lp$ on $\C$, it then follows from Equation~\eqref{eq:LPfromD} that
\begin{equation}\label{eq:prop:equivalentToPelessoniAndVicig:proof:1}
g_0\coloneqq[f_0-\lp(f_0\vert B_0)-\epsilon_0]\ind{B_0}=[f_0-\lp_\desir(f_0\vert B_0)-\epsilon_0]\ind{B_0}\notin\desir.
\end{equation}
Similarly, for all $i\in\{1,\dots,n\}$, Equation~\eqref{eq:LPfromD} implies that there is some $\epsilon_i\geq0$ such that $\epsilon_i\leq\epsilon$ and
\begin{equation}\label{eq:prop:equivalentToPelessoniAndVicig:proof:2}
g_i\coloneqq[f_i-\lp(f_i\vert B_i)+\epsilon_i]\ind{B_i}=[f_i-\lp_\desir(f_i\vert B_i)+\epsilon_i]\ind{B_i}\in\desir.
\end{equation}
Now let $g\coloneqq\lambda_0g_0-\sum_{i=1}^n\lambda_ig_i$ and assume \emph{ex absurdo} that $g\in\desir$. Since $\desir$ is coherent and therefore satisfies~\ref{def:SDG:homo} and~\ref{def:SDG:convex}, it then follows from Equation~\eqref{eq:prop:equivalentToPelessoniAndVicig:proof:2} that
\begin{equation*}
\lambda_0g_0
=
g+\sum_{i=1}^n\lambda_ig_i
=
g+\sum_{\substack{i=1\\\lambda_i\neq0}}^n\lambda_ig_i\in\desir.
\end{equation*}
If $\lambda_0=0$, this implies that $0\in\desir$, which contradicts~\ref{def:SDG:partialloss}. If $\lambda_0>0$, this implies that $g_0\in\desir$ because of~\ref{def:SDG:homo}, which contradicts Equation~\eqref{eq:prop:equivalentToPelessoniAndVicig:proof:1}. Since both cases lead to a contradiction, we conclude that $g\notin\desir$. Since the coherence of $\desir$ implies that $\gamblespos\subseteq\desir$, this allows us to infer that $g\notin\gamblespos$. % and therefore, that $\inf g\leq0$. 
Since $g(x)=0$ for all $x\in\states\setminus B$, this implies that $\inf_{x\in B}g(x)\leq0$. Hence, we find that
\begin{align*}
0\leq
-\inf_{x\in B}g(x)
=\sup_{x\in B}-g(x)
&=\sup_{x\in B}\big(h(x)+\sum_{i=0}^n\lambda_i\epsilon_i\ind{B_i}(x)\big)\\
&\leq\sup_{x\in B}h(x)+\sup_{x\in B}\big(\sum_{i=0}^n\lambda_i\epsilon_i\ind{B_i}(x)\big)
\leq\sup_{x\in B}h(x)+\sum_{i=0}^n\lambda_i\epsilon_i,\vspace{-3pt}
%\leq
%\sup_{x\in B}h(x)+\epsilon\sum_{i=0}^n\lambda_i
\end{align*}
which implies that
\vspace{-3pt}
\begin{equation*}
\sup_{x\in B}h(x)\geq-\sum_{i=0}^n\lambda_i\epsilon_i\geq-\epsilon\sum_{i=0}^n\lambda_i.
\vspace{8pt}
\end{equation*}
Since this is true for every $\epsilon>0$, it follows that $\sup_{x\in B}h(x)\geq0$, as desired.

It remains to prove the `if' part of the statement. So let us assume that $\lp$ is real-valued and that it satisfies Equation~\eqref{eq:prop:equivalentToPelessoniAndVicig}. We need to prove that $\lp$ is coherent.

Let $\A_{\lp}$ and $\E(\lp)$ be defined by Equation~\eqref{eq:AfromLP}. We start by proving that $\E(\lp)$ is a coherent set of desirable gambles on $\states$. Fix any $f\in\E(\lp)$. We then know from Equations~\eqref{eq:posi},~\eqref{eq:natextop} and~\eqref{eq:AfromLP} that\vspace{-4pt}
\begin{equation}
f=\sum_{i=1}^n\lambda_i\ind{B_i}[f_i-\mu_i]+\sum_{j=n+1}^m\lambda_jf_j\geq\sum_{i=1}^n\lambda_i\ind{B_i}[f_i-\mu_i]%\notag\\
%&=\sum_{i=1}^n\lambda_i\ind{B_i}[f_i-\lp(f_i\vert B_i)]+\sum_{i=1}^n\lambda_i\ind{B_i}[\lp(f_i\vert B_i)-\mu_i]
\label{eq:prop:equivalentToPelessoniAndVicig:proof:3}
%\vspace{4pt}
\end{equation}
for some $n\in\natswith$ and $m\in\nats$ such that $n\leq m$, with $\lambda_1,\dots,\lambda_m\in\reals_{>0}$, $f_{n+1},\dots,f_m\in\gamblespos$, and $(f_1,B_1),\dots,(f_n,B_n)\in\C$ and $\mu_1,\dots,\mu_n\in\reals$ such that $\mu_i<\lp(f_i\vert B_i)$ for all $i\in\{1,\dots,n\}$. %, which implies that
%\vspace{-5pt}
% \begin{equation}
% f\geq\sum_{i=1}^n\lambda_i\ind{B_i}[f_i-\mu_i]
% \geq\sum_{i=1}^n\lambda_i\ind{B_i}[f_i-\lp(f_i\vert B_i)]%+\sum_{i=1}^n\lambda_i\ind{B_i}[\lp(f_i\vert B_i)-\mu_i]
% \label{eq:prop:equivalentToPelessoniAndVicig:proof:3}
% \vspace{2pt}
% \end{equation}
We consider two cases: $f\in\gamblespos$ and $f\notin\gamblespos$. 
If $f\in\gamblespos$, then $f\not\leq0$. If $f\notin\gamblespos$, then $n\neq0$. Therefore, if we let $A\coloneqq\cup_{i=1}^nB_i\neq\emptyset$, it follows from Equations~\eqref{eq:prop:equivalentToPelessoniAndVicig:proof:3} and~\eqref{eq:prop:equivalentToPelessoniAndVicig} that
\begin{align*}
\sup_{x\in A}f(x)
&\geq\sup_{x\in A}\Big(\sum_{i=1}^n\lambda_i\ind{B_i}(x)[f_i(x)-\mu_i]\Big)\\
&=\sup_{x\in A}\Big(\sum_{i=1}^n\lambda_i\ind{B_i}(x)[f_i(x)-\lp(f_i\vert B_i)]+\sum_{i=1}^n\lambda_i\ind{B_i}(x)[\lp(f_i\vert B_i)-\mu_i]\Big)\\
&\geq\sup_{x\in A}\Big(\sum_{i=1}^n\lambda_i\ind{B_i}(x)[f_i(x)-\lp(f_i\vert B_i)]\Big)+\inf_{x\in A}\Big(\sum_{i=1}^n\lambda_i\ind{B_i}(x)[\lp(f_i\vert B_i)-\mu_i]\Big)\\
&\geq\inf_{x\in A}\Big(\sum_{i=1}^n\lambda_i\ind{B_i}(x)[\lp(f_i\vert B_i)-\mu_i]\Big)
\geq\min_{1\leq i\leq n}\lambda_i[\lp(f_i\vert B_i)-\mu_i]>0,\\[-13pt]
\end{align*}
which implies that $f\not\leq0$. Hence, in both cases, we find that $f\not\leq0$. Since $f\in\E(\lp)$ is arbitrary, this implies that $\E(\lp)$ satisfies~\ref{def:SDG:partialloss}. Since $\E(\lp)\coloneqq\E(\A_{\lp})$, it now follows from Lemma~\ref{lemma:coherenceiffLP4} that $\E(\lp)$ is a coherent set of desirable gambles on $\states$.

In the remainder of this proof, we will show that $\lp_{\E(\lp)}$ coincides with $\lp$ on $\C$. Since $\E(\lp)$ is a coherent set of desirable gambles on $\states$, Definition~\ref{def:cohlp} then implies that $\lp$ is coherent, as desired. So fix any $(f,B)\in\C$. We need to prove that $\lp(f\vert B)=\lp_{\E(\lp)}(f\vert B)$. However, since Equation~\eqref{eq:AfromLP} implies that $[f-\mu]\ind{B}\in\E(\lp)$ for all $\mu<\lp(f\vert B)$, it follows trivially from Equation~\eqref{eq:LPfromD} that $\lp(f\vert B)\leq\lp_{\E(\lp)}(f\vert B)$. Therefore, it remains to prove that $\lp(f\vert B)\geq\lp_{\E(\lp)}(f\vert B)$.

Consider any $\mu\in\reals$ such that $[f-\mu]\ind{B}\in\E(\lp)$. We then know from Equations~\eqref{eq:posi},~\eqref{eq:natextop} and~\eqref{eq:AfromLP} that\vspace{-4pt}
\begin{equation}
[f-\mu]\ind{B}=\sum_{i=1}^n\lambda_i\ind{B_i}[f_i-\mu_i]+\sum_{j=n+1}^m\lambda_jf_j\geq\sum_{i=1}^n\lambda_i\ind{B_i}[f_i-\mu_i]%\notag\\
%&=\sum_{i=1}^n\lambda_i\ind{B_i}[f_i-\lp(f_i\vert B_i)]+\sum_{i=1}^n\lambda_i\ind{B_i}[\lp(f_i\vert B_i)-\mu_i]
\label{eq:prop:equivalentToPelessoniAndVicig:proof:4}
%\vspace{4pt}
\end{equation}
for some $n\in\natswith$ and $m\in\nats$ such that $n\leq m$, with $\lambda_1,\dots,\lambda_m\in\reals_{>0}$, $f_{n+1},\dots,f_m\in\gamblespos$, and $(f_1,B_1),\dots,(f_n,B_n)\in\C$ and $\mu_1,\dots,\mu_n\in\reals$ such that $\mu_i<\lp(f_i\vert B_i)$ for all $i\in\{1,\dots,n\}$. %We consider two cases: $[f-\mu]\ind{B}\in\gamblespos$ and $[f-\mu]\ind{B}\notin\gamblespos$. 
%If $[f-\mu]\ind{B}\in\gamblespos$, then
% \begin{align*}
% \mu
% \leq\inf_{x\in B}f(x)
% =-\sup_{x\in B}\big(-f(x)\big)
% &=-\sup_{x\in B}\big(-f(x)+\lp(f\vert B)\big)+\lp(f\vert B)\\
% &=-\sup_{x\in B}\big(-\ind{B}[f(x)-\lp(f\vert B)]\big)+\lp(f\vert B)
% \leq\lp(f\vert B),
% \end{align*} 
% using Equation~\eqref{eq:prop:equivalentToPelessoniAndVicig} for the last inequality. If $[f-\mu]\ind{B}\notin\gamblespos$, then $n\neq0$.
%If $f\notin\gamblespos$, then $n\neq0$. 
Therefore, if we let $A\coloneqq B\cup\big(\cup_{i=1}^nB_i\big)\neq\emptyset$, we find that
\begin{align}
&\sup_{x\in A}\Big(\sum_{i=1}^n\lambda_i\ind{B_i}(x)[\mu_i-\lp(f_i\vert B_i)]-\ind{B}(x)[\mu-\lp(f\vert B)]\Big)\notag\\
&\geq
\sup_{x\in A}\Big(\sum_{i=1}^n\lambda_i\ind{B_i}(x)[\mu_i-\lp(f_i\vert B_i)]-\ind{B}(x)[\mu-\lp(f\vert B)]
+
\sum_{i=1}^n\lambda_i\ind{B_i}(x)[f_i(x)-\mu_i]-\ind{B}(x)[f(x)-\mu]
\Big)\notag\\
&=
\sup_{x\in A}\Big(\sum_{i=1}^n\lambda_i\ind{B_i}(x)[f_i(x)-\lp(f_i\vert B_i)]
-\ind{B}(x)[f(x)-\lp(f\vert B)]\Big)\geq0\label{eq:prop:equivalentToPelessoniAndVicig:proof:5}\\[-14pt]\notag
\end{align}
where the first inequality follows from Equation~\eqref{eq:prop:equivalentToPelessoniAndVicig:proof:4} and the last inequality follows from Equation~\eqref{eq:prop:equivalentToPelessoniAndVicig}. 
Since $\lambda_i>0$ and $\mu_i-\lp(f_i\vert B_i)<0$, this implies that $\mu\leq\lp(f\vert B)$. Since this true for every $\mu\in\reals$ such that $[f-\mu]\ind{B}\in\E(\lp)$, it follows from Equation~\eqref{eq:LPfromD} that $\lp_{\E(\lp)}(f\vert B)\leq\lp(f\vert B)$.
\end{proof}

\begin{proof}{\bf of Proposition~\ref{prop:smallestSDGfromLP}~}
Consider any coherent set of desirable gambles $\desir$ on $\states$ such that $\lp_\desir$ coincides with $\lp$ on $\C$. Since $\lp$ is coherent, we know from Definition~\ref{def:cohlp} that there is at least one such set $\desir$. We start by proving that $\E(\lp)\subseteq\desir$.

Fix any $(f,B)\in\C$ and any $\mu<\lp(f\vert B)$. Since $\lp_\desir(f\vert B)=\lp(f\vert B)$, we know that $\mu<\lp_\desir(f\vert B)$, and therefore, it follows from Equation~\eqref{eq:LPfromD} that there is some $\mu^*\in\reals$ such that $[f-\mu^*]\ind{B}\in\desir$ and $\mu<\mu^*\leq\lp_\desir(f\vert B)$. Furthermore, since $\mu^*>\mu$ and $B\neq\emptyset$, we also know that $[\mu^*-\mu]\ind{B}\in\mathcal{G}_{>0}(\states)$, which implies that $[\mu^*-\mu]\ind{B}\in\desir$ because of~\ref{def:SDG:partialgain}. Since $[f-\mu^*]\ind{B}\in\desir$ and $[\mu^*-\mu]\ind{B}\in\desir$, it now follows from~\ref{def:SDG:convex} that $[f-\mu]\ind{B}=[f-\mu^*]\ind{B}+[\mu^*-\mu]\ind{B}\in\desir$. Since this is true for every $(f,B)\in\C$ and $\mu<\lp(f\vert B)$, we infer that $\A_{\lp}\subseteq\desir$, and therefore, because of Lemmas~\ref{lemma:nestedpropsofposandE} and~\ref{lemma:natextDisD}, that $\E(\lp)=\E(\A_{\lp})\subseteq\E(\desir)=\desir$.

Next, since $\desir$ is coherent and $\E(\lp)\subseteq\desir$, it follows from Definition~\ref{def:SDG} that $\E(\lp)$ satisfies~\ref{def:SDG:partialloss}. Therefore, and because $\E(\lp)=\E(\A_{\lp})$, it follows from Lemma~\ref{lemma:coherenceiffLP4} that $\E(\lp)$ is a coherent set of desirable gambles on $\states$. Hence, it remains to prove that $\lp_{\E(\lp)}$ coincides with $\lp$ on $\C$.

Fix any $(f,B)\in\C$. Then on the one hand, since $\E(\lp)\subseteq\desir$, we have that
\begin{equation*}
\lp_{\E(\lp)}(f\vert B)\leq\lp_{\desir}(f\vert B)=\lp(f\vert B).
\end{equation*}
On the other hand, since we know from Equation~\eqref{eq:AfromLP} that $[f-\mu]\ind{B}\in\E(\lp)$ for all $\mu<\lp(f\vert B)$, it follows from Equation~\eqref{eq:LPfromD} that $\lp_{\E(\lp)}(f\vert B)\geq\lp(f\vert B)$. Hence, we find that $\lp_{\E(\lp)}(f\vert B)=\lp(f\vert B)$. Since $(f,B)\in\C$ is arbitrary, this implies that $\lp_{\E(\lp)}$ coincides with $\lp$ on $\C$.
\end{proof}
\vspace{-16pt}

\begin{proposition}\label{prop:naturalextension}
Let $\lp$ be a coherent conditional lower prevision on $\C\subseteq\C(\states)$. Then for any $\C'\subseteq\C(\states)$ such that $\C\subseteq\C'$, the restriction of $\natexLP$ to $\C'$ is the pointwise smallest coherent conditional lower prevision on $\C'$ that coincides with $\lp$ on $\C$.
\end{proposition}

\begin{proof}{\bf of Proposition~\ref{prop:naturalextension}~}
Let $\lp$ be a coherent conditional lower prevision on $\C\subseteq\C(\states)$ and consider any $\C'\subseteq\C(\states)$ such that $\C\subseteq\C'$. Then as we know from Proposition~\ref{prop:smallestSDGfromLP}, $\natexLP$ is a coherent conditional lower on $\C(\states)$ that coincides with $\lp$ on $\C$. Since it follows trivially from Definition~\ref{def:cohlp} that restricting the domain of a coherent conditional lower prevision preserves its coherence, this implies that the restriction of $\natexLP$ to $\C'$ is a coherent conditional lower prevision on $\C'$ that coincides with $\lp$ on $\C$. It remains to show that it is dominated by any other coherent conditional lower prevision on $\C'$ that coincides with $\lp$ on $\C$.

So consider any coherent conditional lower prevision $\lp'$ on $\C'$ that coincides with $\lp$ on $\C$. Because of Definition~\ref{def:cohlp}, this implies that there is a coherent set of desirable gambles $\desir$ on $\states$ such that $\lp_\desir$ coincides with $\lp'$ on $\C'$. Since this clearly implies that $\lp_\desir$ coincides with $\lp$ on $\C$, it now follows from Proposition~\ref{prop:smallestSDGfromLP} that $\E(\lp)\subseteq\desir$, which implies that $\natexLP=\lp_{\E(\lp)}\leq\lp_\desir$. Hence, since $\lp_\desir$ coincides with $\lp'$ on $\C'$, we find that $\natexLP$ is dominated by $\lp'$ on $\C'$, as desired.
\end{proof}
\vspace{-6pt}

\begin{proof}{\bf of Proposition~\ref{prop:naturalextension:full}~}
Immediate consequence of Proposition~\ref{prop:naturalextension}.
\end{proof}
\vspace{-16pt}

\begin{proposition}\label{prop:propertiesofLP}
Let $\lp$ be a coherent conditional lower prevision on $\C\subseteq\C(\states)$. Then for any two gambles $f,g\in\gambles$, any two events \mbox{$A,B\in\nonemptypower$}, any real number $\lambda\in\reals$ and any sequence of gambles $\{f_n\}_{n\in\nats}\subseteq\gambles$, whenever the involved lower and upper previsions are well-defined, we have that 
\vspace{5pt}

\begin{enumerate}[label=\emph{LP\arabic*:},ref=LP\arabic*]
\item
$\lp(f\vert B)\geq\inf_{x\in B}f(x)$\label{def:lowerprev:bounded}\hfill\emph{[boundedness]}
\item
$\lp(\lambda f\vert B)=\lambda\lp(f\vert B)$ if $\lambda\geq0$\label{def:lowerprev:homo}\hfill\emph{[non-negative homogeneity]}
\item
$\lp(f+g\vert B)\geq \lp(f\vert B)+\lp(g\vert B)$\label{def:lowerprev:superadditive}\hfill\emph{[superadditivity]}
\item
$\lp(\ind{B}[f-\lp(f\vert A\cap B)]\vert\, A)=0$ if $A\cap B\neq\emptyset$\label{def:lowerprev:GBR}\hfill\emph{[generalised Bayes rule]}
\item
$\lim_{n\to\infty}\lp(f_n\vert B)=\lp(f\vert B)$
 if\/ $\lim_{n\to\infty}\sup\abs{f-f_n}=0$
% if $f_n$ converges uniformly to $f$
\label{def:lowerprev:uniformcontinuity}\hfill\emph{[uniform continuity]}
\item
$\lp(f+\lambda\vert B)=\lp(f\vert B)+\lambda$\label{def:lowerprev:constantadditivity}\hfill\emph{[constant additivity]}
\item
$\lp(f\vert B)\leq-\lp(-f\vert B)=\overline{P}(f\vert B)$\label{def:lowerprev:lowerbelowupper}%\hfill\emph{[constant additivity]}
\end{enumerate}
\vspace{2pt}
\end{proposition}
%\vspace{-6pt}

\begin{proof}{\bf of Proposition~\ref{prop:propertiesofLP}~}
Let $\lp$ be a coherent conditional lower prevision on $\C\subseteq\C(\states)$ and let $\natexLP$ be its natural extension to $\C(\states)$. We will prove that $\natexLP$ satisfies~\ref{def:lowerprev:bounded}--\ref{def:lowerprev:lowerbelowupper}. Since we know from Proposition~\ref{prop:naturalextension} that $\natexLP$ coincides with $\lp$ on $\C$, this then implies that $\lp$ satisfies~\ref{def:lowerprev:bounded}--\ref{def:lowerprev:lowerbelowupper} on its domain---that is, whenever the expressions are well-defined.

Since we know from Proposition~\ref{prop:naturalextension} that $\natexLP$ is coherent, it follows from Proposition~\ref{prop:equivalentToPelessoniAndVicig} that $\natexLP$ is real-valued and satisfies Equation~\eqref{eq:prop:equivalentToPelessoniAndVicig}, which means that it satisfies the notion of Williams coherence that is considered in~\cite{Pelessoni:2009co} and~\cite{Williams:2007eu}. It therefore follows from~\cite[Theorem~2]{Pelessoni:2009co} or~\cite[(A1*)--(A4*)]{Williams:2007eu} that $\natexLP$ satisfies~\ref{def:lowerprev:bounded}--\ref{def:lowerprev:GBR}. Consider now any $B\in\nonemptypower$. Since the operator $\natexLP(\cdot\vert B)\colon\gambles\to\reals$ satisfies~\ref{def:lowerprev:bounded}--\ref{def:lowerprev:superadditive}, it is a coherent lower prevision in the sense of Walley. Therefore, it follows from~\cite[Section~2.6.1]{Walley:1991vk} that $\natexLP(\cdot\vert B)$ satisfies~\ref{def:lowerprev:uniformcontinuity}--\ref{def:lowerprev:lowerbelowupper}. Since this is true for every $B\in\nonemptypower$, it follows that $\natexLP$ satisfies~\ref{def:lowerprev:uniformcontinuity}--\ref{def:lowerprev:lowerbelowupper} as well.
\end{proof}
\vspace{-16pt}

% *** proof below is based on the ideas in the precise version of the proof in~\cite{Williams:2007eu}) ***

% *** (claimed in~\cite{Vicig:2007gs} that this can be easily derived from the results in~\cite{Williams:2007eu}, however, I don't see it; claimed in~\cite{Pelessoni:2009co} that it is proved in~\cite{Williams:2007eu}) *** For unbounded gambles, these results can be found in the book of Matthias en Gert!!! So... essentially, all of this stuff should be regarded as known! Our proofs are provided for the sake of completeness and self-containedness.

\begin{proposition}\label{prop:cohlpifffouraxioms}
Consider a set of events $\B\subseteq\nonemptypower$ that is closed under finite unions and let $\mathcal{F}\subseteq\gambles$  be a linear space of gambles such that $\ind{B}f\in\mathcal{F}$ and $\ind{B}\in\mathcal{F}$ for every $f\in\mathcal{F}$ and $B\in\B$. Now let $\C\coloneqq\{(f,B)\colon f\in\mathcal{F},B\in\B\}$. 
Then a conditional lower prevision $\lp$ on $\C$ is coherent if and only if it is real-valued and satisfies~\emph{\ref{def:lowerprev:bounded}--\ref{def:lowerprev:GBR}}. 
\end{proposition}

\begin{proof}{\bf of Proposition~\ref{prop:cohlpifffouraxioms}~}
If $\lp$ is coherent, we know from Proposition~\ref{prop:equivalentToPelessoniAndVicig} that $\lp$ is real-valued and from Proposition~\ref{prop:propertiesofLP} that it satisfies~\ref{def:lowerprev:bounded}--\ref{def:lowerprev:GBR}. So assume that $\lp$ is real-valued and satisfies~\ref{def:lowerprev:bounded}--\ref{def:lowerprev:GBR}. We need to prove that $\lp$ is coherent.

Because of Proposition~\ref{prop:equivalentToPelessoniAndVicig}, it suffices to show for all $n\in\natswith$ and all choices of $\lambda_0,\dots,\lambda_n\in\reals_{\geq0}$ and $(f_0,B_0),\dots,(f_n,B_n)\in\mathcal{C}$ that
\begin{equation*}%\label{eq:prop:equivalentToPelessoniAndVicig}
\sup_{x\in B}
\Big(\,
\sum_{i=1}^n
\lambda_i\ind{B_i}(x)
[f_i(x)-\lp(f_i\vert B_i)]
-\lambda_0\ind{B_0}(x)
[f_0(x)-\lp(f_0\vert B_0)]
\Big)
\geq0,
\vspace{6pt}
\end{equation*}
with $B\coloneqq\cup_{i=0}^nB_i$. So let us consider any $n\in\natswith$, $\lambda_0,\dots,\lambda_n\in\reals_{\geq0}$ and $(f_0,B_0),\dots,(f_n,B_n)\in\mathcal{C}$ and let $B\coloneqq\cup_{i=0}^nB_i$. Since $B$ is a finite union of events in $\B$ and because $\B$ is closed under finite unions, we know that $B\in\B$.
 Therefore, and because $\mathcal{F}$ is a linear space such that $\ind{B}f\in\mathcal{F}$ and $\ind{B}\in\mathcal{F}$ for all $f\in\mathcal{F}$ and $B\in\B$, it now follows from~\eqref{def:lowerprev:superadditive} that
\vspace{3pt}
\begin{align*}
&\lp(\lambda_0\ind{B_0}[f_0-\lp(f_0\vert B_0)]\vert B)\\
&\geq
\lp\Big(\lambda_0\ind{B_0}[f_0-\lp(f_0\vert B_0)]-\sum_{i=1}^n\lambda_i\ind{B_i}[f_i-\lp(f_i\vert B_i)]\Big\vert B\Big)
+
\sum_{i=1}^n\lp\big(\lambda_i\ind{B_i}[f_i-\lp(f_i\vert B_i)]\big\vert B\big).
\end{align*}
Hence, since we know from~\ref{def:lowerprev:homo} and~\ref{def:lowerprev:GBR}---and our assumptions on $\mathcal{F}$ and $\B$---that
\begin{equation*}
\lp\big(\lambda_i\ind{B_i}[f_i-\lp(f_i\vert B_i)]\big\vert B\big)
=\lambda_i\lp\big(\ind{B_i}[f_i-\lp(f_i\vert B_i)]\big\vert B\big)=0
~~\text{for all $i\in\{0,\dots,n\}$,}
\end{equation*}
it follows from~\ref{def:lowerprev:bounded} that
\begin{align*}
0&\geq\lp\Big(\lambda_0\ind{B_0}[f_0-\lp(f_0\vert B_0)]-\sum_{i=1}^n\lambda_i\ind{B_i}[f_i-\lp(f_i\vert B_i)]\Big\vert B\Big)\\
&\geq\inf_{x\in B}\Big(\lambda_0\ind{B_0}(x)[f_0(x)-\lp(f_0\vert B_0)]-\sum_{i=1}^n\lambda_i\ind{B_i}(x)[f_i(x)-\lp(f_i\vert B_i)]\Big)\\
&=
-\sup_{x\in B}\Big(\sum_{i=1}^n\lambda_i\ind{B_i}(x)[f_i(x)-\lp(f_i\vert B_i)]-\lambda_0\ind{B_0}(x)[f_0(x)-\lp(f_0\vert B_0)]\Big),
\end{align*}
as desired.
\end{proof}
\vspace{-16pt}

\begin{corollary}\label{corol:cohlpifffouraxioms}
A conditional lower prevision $\lp$ on $\C(\states)$ is coherent if and only if it is real-valued and satisfies~\emph{\ref{def:lowerprev:bounded}--\ref{def:lowerprev:GBR}}. 
\end{corollary}
\begin{proof}{\bf of Corollary~\ref{corol:cohlpifffouraxioms}~}
Immediate consequence of Proposition~\ref{prop:cohlpifffouraxioms}.
\end{proof}
\vspace{-16pt}

%Similarly for an unconditional lower prevision. Discuss various other special cases such as lower probabilities, ...

%\subsection{Conditional linear previsions and probabilities}

\begin{definition}[Conditional prevision]\label{def:prev}
A conditional prevision $\pr$ on $\C\subseteq\C(\states)$ is a conditional lower prevision on $\C$ that is self-conjugate, in the sense that 
%\vspace{-3pt}
\begin{equation}
(-f,B)\in\C
~~\text{and}~~
\pr(f\vert B)=-\pr(-f\vert B)
~~\text{for all $(f,B)\in\C$.}
\label{eq:self-conjugate}
\vspace{3pt}
\end{equation}
%for all $f\in\gambles$ and $B\in\nonemptypower$.
\end{definition}

\begin{definition}[Conditional linear prevision]\label{def:linearprev}
A conditional linear prevision $\pr$ on $\C\subseteq\C(\states)$ is a coherent conditional prevision on $\C$.
\end{definition}
% \begin{definition}[Full conditional linear prevision]\label{def:fulllinearprev}
% A full conditional linear prevision on $\C(\states)$ is a conditional linear prevision whose domain is equal to $\C(\states)$.
% \end{definition}

\begin{proposition}\label{prop:propertiesofP}
Let $\pr$ be a conditional linear prevision on $\C\subseteq\C(\states)$. Then for any two gambles $f,g\in\gambles$, any two events \mbox{$A,B\in\nonemptypower$}, any real number $\lambda\in\reals$ and any sequence of gambles $\{f_n\}_{n\in\nats}\subseteq\gambles$, whenever the involved previsions are well-defined, we have that 
\vspace{5pt}

\begin{enumerate}[label=\emph{P\arabic*:},ref=P\arabic*]
\item
$\pr(f\vert B)\geq\inf_{x\in B}f(x)$\label{def:prev:bounded}\hfill\emph{[boundedness]}
\item
$\pr(\lambda f\vert B)=\lambda\pr(f\vert B)$\label{def:prev:homo}\hfill\emph{[homogeneity]}
\item
$\pr(f+g\vert B)=\pr(f\vert B)+\pr(g\vert B)$\label{def:prev:additive}\hfill\emph{[additivity]}
\item
$\pr(\ind{B}f\vert A)=\pr(f\vert A\cap B)\pr(B\vert A)$ if $A\cap B\neq\emptyset$\label{def:prev:GBR}\hfill\emph{[Bayes rule]}
\item
$\lim_{n\to\infty}\pr(f_n\vert B)=\pr(f\vert B)$
 if\/ $\lim_{n\to\infty}\sup\abs{f-f_n}=0$
% if $f_n$ converges uniformly to $f$
\label{def:prev:uniformcontinuity}\hfill\emph{[uniform continuity]}
\end{enumerate}
%\vspace{2pt}
\end{proposition}

\begin{proof}{\bf of Proposition~\ref{prop:propertiesofP}~}
Because of definitions~\ref{def:prev} and~\ref{def:linearprev}, we know that $\pr$ is a coherent conditional lower prevision on $\C$ that satisfies Equation~\eqref{eq:self-conjugate}. Due to Proposition~\ref{prop:propertiesofLP}, this implies that $\pr$ satisfies~\ref{def:lowerprev:bounded}--\ref{def:lowerprev:uniformcontinuity}. \ref{def:prev:bounded} and~\ref{def:prev:uniformcontinuity} follow trivially from~\ref{def:lowerprev:bounded} and~\ref{def:lowerprev:uniformcontinuity}, respectively. \ref{def:prev:homo} holds because
\begin{equation*}
P(\lambda f\vert B)
=
\begin{cases}
\lambda P(f\vert B)&\text{~if $\lambda\geq0$}\\
-\lambda P(-f\vert B)&\text{~if $\lambda\leq0$}
\end{cases}
~~=\lambda P(f\vert B)
\end{equation*}
where the first equality follows from~\ref{def:lowerprev:homo} and the second one follows from Equation~\eqref{eq:self-conjugate}.
\ref{def:prev:additive} holds because
\begin{equation*}
P(f\vert B)+P(g\vert B)\leq P(f+g\vert B)=-P(-f-g\vert B)\leq-P(-f\vert B)-P(-g\vert B)=P(f\vert B)+P(g\vert B),
\end{equation*}
where the inequalities follow from~\ref{def:lowerprev:superadditive} and the equalities follow from Equation~\eqref{eq:self-conjugate}.
Finally, \ref{def:prev:GBR} holds because
\begin{equation*}
\pr(\ind{B}f\vert A)-\pr(f\vert A\cap B)\pr(B\vert A)
=\pr(\ind{B}f\vert A)-\pr(f\vert A\cap B)\pr(\ind{B}\vert A)
%=\pr(\ind{B}f)+\pr(-\ind{B}\pr(f\vert A\cap B)\vert A)
=\pr(\ind{B}[f-\pr(f\vert A\cap B)]\vert A)=0
\end{equation*}
where second equality follows from~\ref{def:prev:homo} and~\ref{def:prev:additive} and the third equality follows from~\ref{def:lowerprev:GBR}.
\end{proof}
\vspace{-16pt}

\begin{proposition}\label{prop:linearPifffouraxioms}
Consider a set of events $\B\subseteq\nonemptypower$ that is closed under finite unions and let $\mathcal{F}\subseteq\gambles$  be a linear space of gambles such that $\ind{B}f\in\mathcal{F}$ and $\ind{B}\in\mathcal{F}$ for every $f\in\mathcal{F}$ and $B\in\B$. Now let $\C\coloneqq\{(f,B)\colon f\in\mathcal{F},B\in\B\}$. 
Then a conditional prevision $\pr$ on $\C$ is a conditional linear prevision on $\C$ if and only if it is real-valued and satisfies~\emph{\ref{def:prev:bounded}--\ref{def:prev:GBR}}. 
\end{proposition}
%\vspace{-6pt}
\begin{proof}{\bf of Proposition~\ref{prop:linearPifffouraxioms}~}
If $\pr$ is a conditional linear prevision on $\C$, we know from Proposition~\ref{prop:cohlpifffouraxioms} that $\pr$ is real-valued and from Proposition~\ref{prop:propertiesofP} that it satisfies~\ref{def:prev:bounded}--\ref{def:prev:GBR}. So assume that $\pr$ is real-valued and satisfies~\ref{def:prev:bounded}--\ref{def:prev:GBR}. We need to prove that $\pr$ is a conditional linear prevision on $\C$.

Since $\pr$ satisfies~\ref{def:prev:bounded}--\ref{def:prev:additive}, it clearly satisfies~\ref{def:lowerprev:bounded}--\ref{def:lowerprev:superadditive} as well. $\pr$ also satisfies~\ref{def:lowerprev:GBR} because, for all $f\in\mathcal{F}$ and \mbox{$A,B\in\nonemptypower$} such that $A\in\B$ and $\emptyset\neq A\cap B\in\B$, it follows from~\ref{def:prev:homo}--\ref{def:prev:GBR} that
\begin{equation*}
\pr(\ind{B}[f-\pr(f\vert A\cap B)]\vert A)
=\pr(\ind{B}f\vert A)-\pr(f\vert A\cap B)\pr(\ind{B}\vert A)
=
\pr(\ind{B}f\vert A)-\pr(f\vert A\cap B)\pr(B\vert A)=0.
\end{equation*}
Since $\pr$ is real-valued and satisfies~\ref{def:lowerprev:bounded}--\ref{def:lowerprev:GBR}, and because we know from Definition~\ref{def:prev} that $\pr$ is a conditional lower prevision on $\C$, Proposition~\ref{prop:cohlpifffouraxioms} now implies that $\pr$ is coherent. Therefore, it follows from Definition~\ref{def:linearprev} that $\pr$ is a conditional linear prevision on $\C$.
\end{proof}
\vspace{-16pt}

\begin{corollary}\label{corol:fulllinearPifffouraxioms}
A conditional prevision $\pr$ on $\C(\states)$ is a conditional linear prevision on $\C(\states)$ if and only if it is real-valued and satisfies~\emph{\ref{def:prev:bounded}--\ref{def:prev:GBR}}. 
\end{corollary}

\begin{proof}{\bf of Corollary~\ref{corol:fulllinearPifffouraxioms}~}
Immediate consequence of Proposition~\ref{prop:linearPifffouraxioms}.
\end{proof}
\vspace{-16pt}

\begin{lemma}\label{lemma:fulllinearlpifffouraxioms:withsup}
A conditional prevision on $\C(\states)$ is a conditional linear prevision on $\C(\states)$ if and only if it is real-valued and satisfies~\ref{def:prev:boundedbysup} and~\ref{def:prev:homo}--\ref{def:prev:GBR}, with
\begin{enumerate}[label=P\arabic*':,ref=P\arabic*']
\item
$\pr(f\vert B)\leq\sup_{x\in B}f(x)$ for all $(f,B)\in\C(\states)$.\label{def:prev:boundedbysup}
\end{enumerate}
\end{lemma}
\begin{proof}{\bf of Lemma~\ref{lemma:fulllinearlpifffouraxioms:withsup}~}
Since $\pr$ is a conditional prevision and therefore satisfies Equation~\eqref{eq:self-conjugate}, we see that $\pr$ satisfies~\ref{def:prev:bounded} if and only if it satisfies~\ref{def:prev:boundedbysup}. Therefore, the result follows from Corollary~\ref{corol:fulllinearPifffouraxioms}.
\end{proof}
\vspace{-16pt}

%*** Making abstraction of the fact that we have a lower rather than upper prevision, this was essentially already proved by Williams. We can easily derive our result from his. ***

\begin{proposition}\label{prop:lowerenvelope}
A conditional lower prevision $\lp$ on $\C\subseteq\C(\states)$ is coherent if and only if there is a non-empty set $\mathbb{P}$ of conditional linear previsions on $\C(\states)$ such that
\begin{equation}\label{eq:prop:lowerenvelope}
\lp(f\vert B)=\inf\{P(f\vert B)\colon P\in\mathbb{P}\}
~~\text{for all $(f,B)\in\C$.}
\end{equation}
The same is true if the infimum in this expression is replaced by a minimum.
\end{proposition}

\begin{proof}{\bf of Proposition~\ref{prop:lowerenvelope}~}
Let $\lp$ be a coherent conditional lower prevision on $\C\subseteq\C(\states)$.

We first prove the `only if' part of the statement. In order to do that, we let $\natexLP$ be the natural extension of $\lp$ to $\C(\states)$, and let $\natexUP$ be the conditional upper prevision that corresponds to $\natexUP$. We will prove that there is a non-empty set $\mathbb{P}$ of conditional linear previsions on $\C(\states)$ such that
\begin{equation}\label{eq:prop:lowerenvelope:1}
\natexLP(f\vert B)=\min\{P(f\vert B)\colon P\in\mathbb{P}\}
%=\inf\{P(f\vert B)\colon P\in\mathbb{P}\}
~~\text{for all $(f,B)\in\C(\states)$}.
\end{equation}
Since we know from Proposition~\ref{prop:naturalextension:full} that $\natexLP$ coincides with $\lp$ on $\C$, this then clearly implies the `only if' part of the statement.

Since we know from Proposition~\ref{prop:naturalextension:full} that $\natexLP$ is coherent, it follows from Proposition~\ref{prop:equivalentToPelessoniAndVicig} that $\natexLP$ is real-valued and satisfies Equation~\eqref{eq:prop:equivalentToPelessoniAndVicig}. Therefore, for all $n\in\natswith$ and all choices of $\lambda_0,\dots,\lambda_n\in\reals_{\geq0}$ and $(f_0,B_0),\dots,(f_n,B_n)\in\mathcal{C}(\states)$, if we let $B\coloneqq\cup_{i=0}^nB_i$, we find that
\vspace{-4pt}
\begin{multline*}
\sup_{x\in B}
\Big(
\sum_{i=1}^n
\lambda_i\ind{B_i}(x)
[(-f_i(x))-\natexLP(-f_i\vert B_i)]
-
\lambda_0\ind{B_0}(x)
[(-f_0(x))-\natexLP(-f_0\vert B_0)]
\Big)\\
=\sup_{x\in B}
\Big(
\lambda_0\ind{B_0}(x)
[f_0(x)-\natexUP(f_0\vert B_0)]
-\sum_{i=1}^n
\lambda_i\ind{B_i}(x)
[f_i(x)-\natexUP(f_i\vert B_i)]
\Big)
\geq0.
\end{multline*}
Since this means that $\natexUP$ satisfies condition (A*) in~\cite{Williams:2007eu}, it now follows from~\cite[Theorem~2, Definition~2 and Proposition~6]{Williams:2007eu} and Lemma~\ref{lemma:fulllinearlpifffouraxioms:withsup} that there is a non-empty set $\mathbb{P}$ of conditional linear previsions on $\C(\states)$ such that
\begin{equation}\label{eq:prop:lowerenvelope:2}
\natexUP(f\vert B)=\sup\{P(f\vert B)\colon P\in\mathbb{P}\}=\max\{P(f\vert B)\colon P\in\mathbb{P}\}
~~\text{for all $(f,B)\in\C(\states)$}.
\end{equation}
The first equality corresponds to~\cite[Theorem~2]{Williams:2007eu}; the second equality is not stated in~\cite[Theorem~2]{Williams:2007eu} itself, but follows from the end of its proof. \cite[Definition~2 and Proposition~6]{Williams:2007eu} and Lemma~\ref{lemma:fulllinearlpifffouraxioms:withsup} are needed solely for the purpose of establishing that what Williams calls a conditional prevision in~\cite[Theorem~2]{Williams:2007eu} is equivalent to what we here call a conditional linear prevision on $\C(\states)$. Equation~\eqref{eq:prop:lowerenvelope:1} now follows because, for all $(f,B)\in\C(\states)$, 
\begin{align*}
\natexLP(f\vert B)
=-\natexUP(-f\vert B)
&=-\max\{P(-f\vert B)\colon P\in\mathbb{P}\}\\
&=-\max\{-P(f\vert B)\colon P\in\mathbb{P}\}
=\min\{P(f\vert B)\colon P\in\mathbb{P}\},
\end{align*}
where the second equality follows from Equation~\eqref{eq:prop:lowerenvelope:2} and the third equality follows from Equation~\eqref{eq:self-conjugate}.

We end by proving the `if' part of the statement. So assume that there is some non-empty set $\mathbb{P}$ of conditional linear previsions on $\C(\states)$ that satisfies Equation~\eqref{eq:prop:lowerenvelope}. We will prove that $\lp$ is real-valued and that, for all $n\in\natswith$ and all choices of $\lambda_0,\dots,\lambda_n\in\reals_{\geq0}$ and $(f_0,B_0),\dots,(f_n,B_n)\in\mathcal{C}$,
\begin{equation}\label{eq:prop:lowerenvelope:3}
\sup_{x\in B}
\Big(\,
\sum_{i=1}^n
\lambda_i\ind{B_i}(x)
[f_i(x)-\lp(f_i\vert B_i)]
-\lambda_0\ind{B_0}(x)
[f_0(x)-\lp(f_0\vert B_0)]
\Big)
\geq0,
\vspace{6pt}
\end{equation}
with $B\coloneqq\cup_{i=0}^nB_i$. Proposition~\ref{prop:equivalentToPelessoniAndVicig} then implies that $\lp$ is coherent.

Let us first prove that $\lp$ is real-valued. Fix any $(f,B)\in\C$. For every $P\in\mathbb{P}$, it then follows from Proposition~\ref{prop:propertiesofP} and Lemma~\ref{lemma:fulllinearlpifffouraxioms:withsup} that $\inf_{x\in B}\leq P(f\vert B)\leq\sup_{x\in B}f(x)$. Hence, since $\mathbb{P}$ is non-empty, it follows from Equation~\eqref{eq:prop:lowerenvelope} that $\inf_{x\in B}\leq\lp(f\vert B)\leq\sup_{x\in B}f(x)$. Since $f$ is a gamble and therefore by definition bounded, this implies that $\lp(f\vert B)$ is real-valued. Since this is true for every $(f,B)\in\C$, it follows that $\lp$ is real-valued.

Finally, fix any $n\in\natswith$, any $\lambda_0,\dots,\lambda_n\in\reals_{\geq0}$ and $(f_0,B_0),\dots,(f_n,B_n)\in\mathcal{C}$, let $B\coloneqq\cup_{i=0}^nB_i$ and consider any $\epsilon>0$. It then follows from Equation~\eqref{eq:prop:lowerenvelope} that there is a conditional linear prevision $P\in\mathbb{P}$ on $\C(\states)$ such that $\lambda_0 P(f_0\vert B_0)\leq\lambda_0\lp(f_0\vert B_0)+\epsilon$. Furthermore, for any $i\in\{1,\dots,n\}$, Equation~\eqref{eq:prop:lowerenvelope} also implies that $P(f_i\vert B_i)\geq\lp(f_i\vert B_i)$. Hence, we find that
\begin{align*}
\sup_{x\in B}
\Big(\,
\sum_{i=1}^n
\lambda_i\ind{B_i}(x)
&[f_i(x)-\lp(f_i\vert B_i)]
-\lambda_0\ind{B_0}(x)
[f_0(x)-\lp(f_0\vert B_0)]
\Big)\\
&\geq\sup_{x\in B}
\Big(\,
\sum_{i=1}^n
\lambda_i\ind{B_i}(x)
[f_i(x)-\pr(f_i\vert B_i)]
-\lambda_0\ind{B_0}(x)
[f_0(x)-\pr(f_0\vert B_0)]
-\ind{B_0}\epsilon
\Big)\\
&\geq\sup_{x\in B}
\Big(\,
\sum_{i=1}^n
\lambda_i\ind{B_i}(x)
[f_i(x)-\pr(f_i\vert B_i)]
-\lambda_0\ind{B_0}(x)
[f_0(x)-\pr(f_0\vert B_0)]
\Big)-\epsilon\geq-\epsilon,
\end{align*}
where the last inequality follows from Proposition~\ref{prop:equivalentToPelessoniAndVicig} because we know from Definition~\ref{def:linearprev} that $P$ is coherent. Since $\epsilon>0$ is arbitrary, we obtain Equation~\eqref{eq:prop:lowerenvelope:3}, as desired.
\end{proof}
\vspace{-10pt}

\subsection{Proofs and Additional Material for Section~\ref{sec:indnatext}}
\vspace{5pt}

\subsubsection{The Sets of Desirable Gambles Part}
\vspace{5pt}

\begin{proposition}\label{prop:productcoherent:SDG}
$\desir_1\otimes\desir_2$ is a coherent set of desirable gambles on $\states_1\times\states_2$.
\end{proposition}
\begin{proof}{\bf of Proposition~\ref{prop:productcoherent:SDG}~} 
Because of Lemma~\ref{lemma:coherenceiffLP4}, it suffices to prove~\ref{def:SDG:partialloss}. So consider any $f\in\desir_1\otimes\desir_2$ and assume \emph{ex absurdo} that $f\leq0$. We will prove that this leads to a contradiction.

Since $\desir_1$ and $\desir_2$ are coherent, they are closed with respect to positive scaling and finite sums. Therefore, and because $f\in\desir_1\otimes\desir_2=\E
\left(
\A_{1\to2}
\cup
\A_{2\to1}
\right)$, it follows from Equations~\eqref{eq:A12s} and~\eqref{eq:A21s} that
\vspace{2pt}
\begin{equation}\label{eq:fexplicit}
f=\sum_{i\in I}\ind{B_{1,i}}(X_1)f_{2,i}(X_2)+\sum_{j\in J}\ind{B_{2,j}}(X_2)f_{1,j}(X_1)+g,
\end{equation}
with $I$ and $J$ finite---possibly empty---index sets, with $B_{1,i}\in\nonemptypoweron{1}$ and $f_{2,i}\in\desir_2$ for all $i\in I$, with $B_{2,j}\in\nonemptypoweron{2}$ and $f_{1,j}\in\desir_1$ for all $j\in J$, with $g\geq0$, and where $g=0$ is only possible if $\abs{I}+\abs{J}>0$. 

Let us assume~\emph{ex absurdo} that $\abs{I}+\abs{J}=0$. Then on the one hand, since we know that  $g=0$ is only possible if $\abs{I}+\abs{J}>0$, it follows that $g\neq0$. On the other hand, $\abs{I}+\abs{J}=0$ also implies that $I=J=\emptyset$, and therefore, due to Equation~\eqref{eq:fexplicit}, that $f=g$. Since $g\geq0$ and $f\leq0$, this in turn implies that $g=0$, thereby contradicting the fact that $g\neq0$. Hence, it follows that at least one of the two \emph{ex absurdo} assumptions that we have so far made must be wrong. If $f\not\leq0$, then the proof is finished. For that reason, in the remainder of the proof, we can assume that $\abs{I}+\abs{J}\neq0$, and therefore, that $\abs{I}+\abs{J}>0$. The only \emph{ex absurdo} assumption that still remains is that $f\leq0$.

Now let $\{B_{1,k}\}_{k\in K}$ be the set consisting of those atoms of the algebra generated by $\{B_{1,i}\}_{i\in I}$ that belong to $\cup_{i\in I}B_{1,i}$ and, for all $k\in K$, let $f_{2,k}\coloneqq\sum_{i\in I\colon B_{1,k}\subseteq B_{1,i}}f_{2,i}$. The following properties are then easily verified. First, since $I$ is finite, $K$ is also finite. Secondly, $\abs{K}=0$ if and only if $\abs{I}=0$. Thirdly, for all $k\in K$, we have that $B_{1,k}\in\nonemptypoweron{1}$ and---since $\desir_1$ is coherent and therefore satisfies~\ref{def:SDG:convex}---that $f_{2,k}\in\desir_1$. Fourthly, $\sum_{k\in K}\ind{B_{1,k}}(X_1)f_{2,k}(X_2)$ is equal to $\sum_{i\in I}\ind{B_{1,i}}(X_1)f_{2,i}(X_2)$. Fiftly, the events in $\{B_{1,k}\}_{k\in K}$ are pairwise disjoint. For this reason, without loss of generality, we can assume the events $\{B_{1,i}\}_{i\in I}$ in Equation~\eqref{eq:fexplicit} to be pairwise disjoint. A completely similar argument leads us to conclude that the events $\{B_{2,j}\}_{j\in J}$ in Equation~\eqref{eq:fexplicit} can be assumed to be pairwise disjoint, again without loss of generality.

%if they are not disjoint, then simply replace them with the induced set of atoms. 
%We now consider two mutually exclusive and exhaustive cases: $\abs{I}+\abs{J}=0$ and $\abs{I}+\abs{J}>0$. 

If $\{B_{1,i}\}_{i\in I}$ is a partition of $\states_1$, then we let $\Y_1\coloneqq I$. Otherwise, we let $\Y_1\coloneqq I\cup\{i^*\}$ and define $B_{1,i^*}\coloneqq\states_1\setminus\cup_{i\in I}B_{1,i}$. Similarly, we let $\Y_2\coloneqq J$ if $\{B_{2,j}\}_{j\in J}$ is a partition of $\states_2$, and let $\Y_2\coloneqq J\cup\{j^*\}$ and $B_{1,j^*}\coloneqq\states_2\setminus\cup_{j\in J}B_{2,j}$ otherwise.
Next, for every $i\in I$, we let $h_{2,i}$ be a gamble on $\Y_2$, defined by
\begin{equation}\label{eq:discretise2}
h_{2,i}(y_2)\coloneqq\sup\{f_{2,i}(x_2)\colon x_2\in B_{2,y_2}\}
~\text{ for all $y_2\in\Y_2$.}
\vspace{4pt}
\end{equation}
Similarly, for every $j\in J$, we let $h_{1,j}$ be a gamble on $\Y_1$, defined by
\begin{equation}\label{eq:discretise1}
h_{1,j}(y_1)\coloneqq\sup\{f_{1,j}(x_1)\colon x_1\in B_{1,y_1}\}
~\text{ for all $y_1\in\Y_1$.}
\vspace{2pt}
\end{equation}
Using these gambles on $\Y_1$ and $\Y_2$, we now construct a real-valued function $h$ on $\Y_1\times\Y_2$, defined by
\begin{equation}\label{eq:hforEcontradiction}
h(y_1,y_2)
\coloneqq
\sum_{i\in I}\ind{i}(y_1)h_{2,i}(y_2)+\sum_{j\in J}\ind{j}(y_2)h_{1,j}(y_1)
%=\ind{I}(y_1)h_{2,y_1}(y_2)+\ind{J}(y_2)h_{1,y_2}(y_1)
\text{~~for all $y_1\in\Y_1$ and $y_2\in\Y_2$}
\vspace{3pt}
\end{equation}
This function is non-positive, in the sense that $h\leq0$. In order to prove that, let us fix any $y_1\in\Y_1$ and $y_2\in\Y_2$. It then follows from Equations~\eqref{eq:discretise2} and~\eqref{eq:discretise1} that
\vspace{4pt}
\begin{equation*}
h(y_1,y_2)
=
\sum_{i\in I}\ind{i}(y_1)\sup_{x_2\in B_{2,y_2}}f_{2,i}(x_2)+\sum_{j\in J}\ind{j}(y_2)\sup_{x_1\in B_{1,y_1}}f_{1,j}(x_1).
\end{equation*}
Since $\ind{i}(y_1)$ can be non-zero for at most one $i\in I$ and $\ind{j}(y_2)$ can be non-zero for at most one $j\in J$, we know that each of the two summations on the right hand side contains at most one non-zero term. The suprema can therefore be moved outside of the summations, yielding
\begin{equation*}
h(y_1,y_2)
=
\sup_{x_1\in B_{1,y_1}}\sup_{x_2\in B_{2,y_2}}\left(\sum_{i\in I}\ind{i}(y_1)f_{2,i}(x_2)+\sum_{j\in J}\ind{j}(y_2)f_{1,j}(x_1)\right).
\end{equation*}
For the next step, we start by observing the following. For any $x_1\in B_{1,y_1}$ and any $i\in I$, since the sets $\{B_{1,i}\}_{i\in I}$ are pairwise disjoint, we know that $x_1\in B_{1,i}$ if and only if $y_1=i$, which implies that $\ind{i}(y_1)=\ind{B_{1,i}}(x_1)$. Similarly, for any $x_2\in B_{2,y_2}$ and any $j\in J$, since the sets $\{B_{2,j}\}_{j\in J}$ are pairwise disjoint, we know that $x_2\in B_{2,j}$ if and only if $y_2=j$, which implies that $\ind{j}(y_2)=\ind{B_{2,j}}(x_2)$. As an immediate consequence, it follows that
\begin{equation*}
h(y_1,y_2)
=
\sup_{x_1\in B_{1,y_1}}\sup_{x_2\in B_{2,y_2}}\left(\sum_{i\in I}\ind{B_{1,i}}(x_1)f_{2,i}(x_2)+\sum_{j\in J}\ind{B_{2,j}}(x_2)f_{1,j}(x_1)\right).
\end{equation*}
Finally, in combination with Equation~\eqref{eq:fexplicit}, this implies that
\begin{equation*}
h(y_1,y_2)
=
\sup_{x_1\in B_{1,y_1}}\sup_{x_2\in B_{2,y_2}}\left(
f(x_1,x_2)-g(x_1,x_2)
\right)
\leq0,
\end{equation*}
where, for the last inequality, we use the fact that $f\leq0$ and $g\geq0$. Since this true for every $y_1\in\Y_1$ and $y_2\in\Y_2$, it follows that $h\leq0$.

Now let $\A_1\coloneqq\{h_{1,j}\colon j\in J\}$ and $\A_2\coloneqq\{h_{2,i}\colon i\in I\}$ and assume~\emph{ex absurdo} that $\mathcal{H}_1\coloneqq\E(\A_1)$ and $\mathcal{H}_2\coloneqq\E(\A_2)$ are coherent sets of desirable gambles on $\Y_1$ and $\Y_2$, respectively. We will prove that that is impossible, by constructing a probability mass function $p$ on $\states_1\times\states_2$ such that the corresponding expectation of $h$ is both non-positive \emph{and}  positive, thereby obtaining a contradiction. In order to do that, we borrow an argument of De Cooman and Miranda (\citeyear[Proof of Proposition~15]{deCooman:2012vba}) that is based on a very useful lemma of them, which, in order to make this paper self-contained, is restated here in Lemma~\ref{lemma:GertandQuique}.

Since $\mathcal{H}_1$ is a coherent set of desirable gambles on $\Y_1$, it follows from Definition \ref{def:SDG}---and \ref{def:SDG:partialloss} in particular---that $0\notin\mathcal{H}_1=\E(\A_1)$. Therefore, and because $\Y_1$ and $J$---and hence also $\A_1$---are finite, it follows from Lemma~\ref{lemma:GertandQuique} that there is a probability mass function $p_1$ on $\Y_1$ such that $p_1(y_1)>0$ for all $y_1\in\Y_1$ and $\sum_{y_1\in\Y_1}p_1(y_1)h_{1,j}(y_1)$ for all $j\in J$. Using a completely analogous argument, we also infer that there is a probability mass function $p_2$ on $\Y_2$ such that $p_2(y_2)>0$ for all $y_2\in\Y_2$ and $\sum_{y_2\in\Y_2}p_2(y_2)h_{2,i}(y_2)$ for all $i\in I$. 

We now let $p$ be the probability mass function on $\Y_1\times\Y_2$ that is defined by $p(y_1,y_2)\coloneqq p_1(y_1)p_2(y_2)$ for all $y_1\in\Y_1$ and $y_1\in\Y_2$, and we let $E_p(h)$ be the expectation of $h$ with respect to $p$, as defined by
\begin{equation*}
E_p(h)
\coloneqq
\sum_{y_1\in\Y_1}
\sum_{y_2\in\Y_2}
p(y_1,y_2)h(y_1,y_2)
=
\sum_{y_1\in\Y_1}
\sum_{y_2\in\Y_2}
p_1(y_1)p_2(y_2)h(y_1,y_2).
\end{equation*}
Then on the one hand, since $h\leq0$, we have that $E_p(h)\leq0$. On the other hand, however, it follows from Equation~\eqref{eq:hforEcontradiction} that
\vspace{-3pt}
\begin{align*}
E_p(h)
&=\sum_{y_1\in\Y_1}
\sum_{y_2\in\Y_2}
p_1(y_1)p_2(y_2)\left(
\sum_{i\in I}\ind{i}(y_1)h_{2,i}(y_2)+\sum_{j\in J}\ind{j}(y_2)h_{1,j}(y_1)
\right)\\
&=\sum_{y_1\in\Y_1}
\sum_{y_2\in\Y_2}
p_1(y_1)p_2(y_2)
\sum_{i\in I}\ind{i}(y_1)h_{2,i}(y_2)+
\sum_{y_1\in\Y_1}
\sum_{y_2\in\Y_2}
p_1(y_1)p_2(y_2)
\sum_{j\in J}\ind{j}(y_2)h_{1,j}(y_1)\\
&=
\sum_{i\in I}
\sum_{y_1\in\Y_1}
p_1(y_1)
\ind{i}(y_1)
\sum_{y_2\in\Y_2}
p_2(y_2)
h_{2,i}(y_2)
+
\sum_{j\in J}
\sum_{y_2\in\Y_2}
p_2(y_2)
\ind{j}(y_2)
\sum_{y_1\in\Y_1}
p_1(y_1)
h_{1,j}(y_1)\\
&=
\sum_{i\in I}
p_1(i)
\sum_{y_2\in\Y_2}
p_2(y_2)
h_{2,i}(y_2)
+
\sum_{j\in J}
p_2(j)
\sum_{y_1\in\Y_1}
p_1(y_1)
h_{1,j}(y_1).
\vspace{4pt}
\end{align*}
For every $i\in I$, it follows from the properties of $p_1$ that the corresponding term in this summation is positive. Similarly, for every $j\in J$, it follows from the properties of $p_2$ that the corresponding term in this summation is strictly positive. Since $\abs{I}+\abs{J}>0$, this implies that $E_p(h)>0$, thereby contradicting the fact that $E_p(h)\leq0$. Hence, it follows that one of the two remaining \emph{ex absurdo} assumptions is wrong. If $f\leq0$, then the proof is finished. Therefore, in the remainder of the proof, we can assume that there is at least one $i\in\{1,2\}$ for which $\mathcal{H}_i$ is incoherent. Without loss of generality, symmetry allows us to assume that $i=1$, that is, that $\mathcal{H}_1$ is incoherent. The only \emph{ex absurdo} assumption that still remains is that $f\leq0$.

Since $\mathcal{H}_1$ is incoherent, it follows from Lemma~\ref{lemma:coherenceiffLP4} that there is some $h^*\in\mathcal{H}_1$ such that $h^*\leq0$. Furthermore, since $h^*\in\mathcal{H}_1$, Equation~\eqref{eq:natextop} implies that $h^*=\lambda g^*+\sum_{j\in J}\lambda_j h_{1,j}$, for some $\lambda\in\realsnonneg$ and $g^*\in\mathcal{G}_{>0}(\Y_1)$ and, for all $j\in J$, some $\lambda_j\in\realsnonneg$, with $\lambda+\sum_{j\in J}\lambda_j>0$. If $\lambda_j=0$ for all $j\in J$, then $\lambda>0$ and $g^*=\nicefrac{1}{\lambda}h^*\leq0$, which is impossible because $g^*\in\mathcal{G}_{>0}(\Y_1)$. Therefore, we know that there is at least one $j\in J$ such that $\lambda_j>0$.

Now let $f_1\coloneqq\sum_{j\in J}\lambda_jf_{1,j}$ and fix any $x_1^*\in\states_1$. Since the events in $\{B_{1,y_1}\}_{y_1\in \Y_1}$ are pairwise disjoint, there will then be a unique $y_1^*\in\Y_1$ such that $x_1^*\in B_{1,y_1^*}$. For this particular choice of $y_1^*$, we then find that
\begin{equation*}
f_1(x_1^*)
=
\sum_{j\in J}\lambda_jf_{1,j}(x_1^*)
\leq
\sum_{j\in J}\lambda_j
\sup_{x_1\in B_{1,y_1^*}}
f_{1,j}(x_1)
=
\sum_{j\in J}\lambda_j
h_{1,j}(y_1^*)
=h^*(y_1^*)-\lambda g^*(y_1^*)\leq0,
\end{equation*}
where the first equality follows from Equation~\eqref{eq:discretise1} and the second inequality follows from the fact that $h^*\leq0$, $\lambda\geq0$ and $g^*\in\mathcal{G}_{>0}(\Y_1)$. Since this is true for every $x_1^*\in\states_1$, we infer that $f_1\leq0$. However, on the other hand, since there is at least one $j\in J$ such that $\lambda_j>0$, and because $f_{1,j}\in\desir_1$ for all $j\in J$, the coherence of $\desir_1$ implies that $f_1\in\desir_1$ and therefore, because of~\ref{def:SDG:partialloss}, that $f_1\not\leq0$. From this contradiction, it follows that one of our \emph{ex absurdo} assumptions must be false. Since the only remaining \emph{ex absurdo} assumption is that $f\leq0$, this concludes the proof.
\end{proof}
\vspace{-16pt}

\begin{lemma}\label{lemma:GertandQuique}\cite[Lemma~2]{deCooman:2012vba}
Let $\Omega$ be a finite set and consider some finite subset $\A$ of $\mathcal{G}(\Omega)$. Then $0\notin\E(\A)$ if and only if there is a probability mass function $p$ on $\Omega$ such that $p(\omega)>0$ for all $\omega\in\Omega$ and $\sum_{\omega\in\Omega}p(\omega)f(\omega)>0$ for all $f\in\A$.
\end{lemma}

\begin{proposition}\label{prop:productindependent:SDG}
$\desir_1\otimes\desir_2$ is an independent product of $\desir_1$ and $\desir_2$.
\end{proposition}
\begin{proof}{\bf of Proposition~\ref{prop:productindependent:SDG}}
For ease of notation, let $\desir\coloneqq\desir_1\otimes\desir_2$. Because of symmetry, it clearly suffices to prove that
\begin{equation*}
(\forall B_2\in\B_2)~~\desir_1=\marg_1(\desir)=\marg_1(\desir\rfloor B_2),
\vspace{6pt}
\end{equation*}
which, since $\marg_1(\desir)=\marg_1(\desir\rfloor\states_2)$, is equivalent to proving that, for all $f_1\in\gambleson{1}$ and $B_2\in\B_2\cup\{\states_2\}$,
\begin{equation*}
f_1(X_1)\ind{B_2}(X_2)\in\desir
~\asa~
f_1\in\desir_1.
\vspace{5pt}
\end{equation*}
Since $f_1\in\desir_1$ implies that $f_1(X_1)\ind{B_2}(X_2)\in\A_{2\to1}\subseteq\desir$ for all $B_1\in\B_2\cup\{\states_2\}$, the converse implication holds trivially. So consider any $f_1\in\gambleson{1}$ and $B_2\in\B_2\cup\{\states_2\}$ such that $f_1(X_1)\ind{B_2}(X_2)\in\desir$. Since we know from Proposition~\ref{prop:productcoherent:SDG} that $\desir$ is coherent, this implies that $f_1\neq0$. It remains to prove that $f_1\in\desir_1$.

Assume \emph{ex absurdo} that $f_1\notin\desir_1$. Then since $f_1\neq0$, $\desir_1^{\bullet}\coloneqq\E(\desir_1\cup\{-f_1\})$ is a coherent set of desirable gambles on $\states_1$ because of Lemma~\ref{lemma:extendD}, and therefore, if we let
\vspace{4pt}
\begin{equation}\label{eq:Abullet}
\A_{2\to1}^{\bullet}
\coloneqq
\left\{
f'_1(X_1)\ind{B'_2}(X_2)
\colon
f'_1\in\desir_1^{\bullet}, 
B'_2\in\B_2\cup\{\states_2\}
\right\},
\vspace{4pt}
\end{equation}
it follows from Proposition~\ref{prop:productcoherent:SDG} that $\desir_1^\bullet\otimes\desir_2\coloneqq\E(\A_{1\to2}\cup\A^\bullet_{2\to1})$ is a coherent set of desirable gambles on $\states_1\times\states_2$. Now on the one hand, since $-f_1\in\desir_1^\bullet$, it follows from Equation~\eqref{eq:Abullet} that $-f_1(X_1)\ind{B_2}(X_2)\in\A_{2\to1}^{\bullet}\subseteq\desir_1^\bullet\otimes\desir_2$. On the other hand, since $\desir_1\subseteq\desir_1^{\bullet}$ implies that $\desir_1\otimes\desir_2\subseteq\desir_1^\bullet\otimes\desir_2$, we infer from $f_1(X_1)\ind{B_2}(X_2)\in\desir$ that $f_1(X_1)\ind{B_2}(X_2)\in\desir_1^\bullet\otimes\desir_2$. Since $\desir_1^\bullet\otimes\desir_2$ is coherent, this implies that
\vspace{3pt}
\begin{equation*}
0=f_1(X_1)\ind{B_2}(X_2)-f_1(X_1)\ind{B_2}(X_2)\in
\desir_1^\bullet\otimes\desir_2,
\vspace{3pt}
\end{equation*}
which contradicts~\ref{def:SDG:partialloss}.
\end{proof}
\vspace{-6pt}

\begin{proof}{\bf of Theorem~\ref{theo:natext:SDG}~}
Since we know from Proposition~\ref{prop:productindependent:SDG} that $\desir_1\otimes\desir_2$ is an independent product of $\desir_1$ and $\desir_2$, it suffices to prove that any other such independent product of $\desir_1$ and $\desir_2$ is a superset of $\desir_1\otimes\desir_2$.

So let $\desir$ be any independent product of $\desir_1$ and $\desir_2$. Definition~\ref{def:subsetindependentproduct} then implies that $\desir$ is coherent and that $\A_{1\to2}
\cup
\A_{2\to1}
\subseteq\desir$. Hence, we find that
\begin{equation*}
\desir_1\otimes\desir_2
=
\E
\left(
\A_{1\to2}
\cup
\A_{2\to1}
\right)
\subseteq\E(\desir)=\desir, 
\end{equation*}
where the inclusion follows from Lemma~\ref{lemma:nestedpropsofposandE} and the final equality from Lemma~\ref{lemma:natextDisD}.
\end{proof}
\vspace{-13pt}

\subsubsection{The Conditional Lower Previsions Part}
\vspace{5pt}

\begin{proposition}\label{prop:productcoherent:LP}
$\lp_1\otimes\lp_2$ is a coherent conditional probability on $\C(\states_1\times\states_2)$.
\end{proposition}
%\vspace{-6pt}

\begin{proof}{\bf of Proposition~\ref{prop:productcoherent:LP}~}
For all $i\in\{1,2\}$, since $\lp_i$ is a coherent conditional lower prevision on $\C_i$, it follows from Proposition~\ref{prop:smallestSDGfromLP} that $\E(\lp_i)$ is a coherent set of desirable gambles on $\states_i$. Therefore, Proposition~\ref{prop:productcoherent:SDG} implies that $\E(\lp_1)\otimes\E(\lp_2)$ is a coherent set of desirable gambles on $\states_1\times\states_2$. The result now follows from Definition~\ref{def:cohlp}.
\end{proof}
\vspace{-16pt}

\begin{proposition}\label{prop:independent:localnatex}
Consider two indexes $i$ and $j$ such that $\{i,j\}=\{1,2\}$. Then for any $f_i\in\gambleson{i}$ and $B_i\in\nonemptypoweron{i}$ and any $B_j\in\B_j$, we have that
\begin{equation}\label{eq:prop:independent:localnatex}
(\lp_1\otimes\lp_2)(f_i\vert B_i\cap B_j)
=
(\lp_1\otimes\lp_2)(f_i\vert B_i)
=
\natexLP_i(f_i\vert B_i).
\end{equation}
\end{proposition}
%\vspace{-6pt}

%By combining the two results above with Proposition~\ref{prop:naturalextension}, it follows rather easily that the restriction of $\lp_1\otimes\lp_2$ to $\C$ is an independent product of $\lp_1$ and $\lp_2$. Our next result establishes that it is furthermore the smallest such independent product.

\begin{proof}{\bf of Proposition~\ref{prop:independent:localnatex}~}
For all $i\in\{1,2\}$, since $\lp_i$ is a coherent conditional lower prevision on $\C_i$, it follows from Proposition~\ref{prop:smallestSDGfromLP} that $\E(\lp_i)$ is a coherent set of desirable gambles on $\states_i$. Therefore, we infer from Proposition~\ref{prop:productindependent:SDG} that $\E(\lp_1)\otimes\E(\lp_2)$ is an independent product of $\E(\lp_1)$ and $\E(\lp_2)$.
For ease of notation, we now let $\lp\coloneqq\lp_1\otimes\lp_2$ and $\desir\coloneqq\E(\lp_1)\otimes\E(\lp_2)$. As we know from Equation~\eqref{eq:indnatex:LP}, $\lp$ is then equal to $\lp_\desir$. Furthermore, since $\desir$ is an independent product of $\E(\lp_1)$ and $\E(\lp_2)$, we know that $\desir$ is epistemically independent and that it has $\E(\lp_1)$ and $\E(\lp_2)$ as its marginals.

We are now ready to prove Equation~\eqref{eq:prop:independent:localnatex}. In order to do that, we fix any two indexes $i$ and $j$ such that $\{i,j\}=\{1,2\}$, any $f_i\in\gambleson{i}$ and $B_i\in\nonemptypoweron{i}$ and any $B_j\in\B_j$. We start by proving the first equality. Since $\desir$ is epistemically independent, we know that
\begin{align*}
[f_i-\mu]\ind{B_i}\in\desir
&\asa
[f_i-\mu]\ind{B_i}\in\marg_i(\desir)\\
&\asa
[f_i-\mu]\ind{B_i}\in\marg_i(\desir\vert B_j)
\asa
[f_i-\mu]\ind{B_i}\ind{B_j}\in\desir
\asa
[f_i-\mu]\ind{B_i\cap B_j}\in\desir
\end{align*}
for all $\mu\in\reals$,
and therefore, we find that
\begin{equation*}
\lp(f_i\vert B_i)
=\sup\big\{\mu\in\reals\colon[f_i-\mu]\ind{B_i}\in\desir\big\}
=\sup\big\{\mu\in\reals\colon[f_i-\mu]\ind{B_i\cap B_j}\in\desir\big\}
=
\lp(f_i\vert B_i\cap B_j).
\end{equation*}
Next, we prove the second equality of Equation~\eqref{eq:prop:independent:localnatex}. Since $\desir$ has $\E(\lp_1)$ and $\E(\lp_2)$ as its marginals, we know that
\vspace{-6pt}
\begin{equation*}
[f_i-\mu]\ind{B_i}\in\desir
\asa
[f_i-\mu]\ind{B_i}\in\marg_i(\desir)
\asa
[f_i-\mu]\ind{B_i}\in\E(\lp_i)
\vspace{3pt}
\end{equation*}
for all $\mu\in\reals$, and therefore, we find that
\begin{equation*}
\lp(f_i\vert B_i)
=
\sup\big\{\mu\in\reals\colon[f_i-\mu]\ind{B_i}\in\desir\big\}
=
\sup\big\{\mu\in\reals\colon[f_i-\mu]\ind{B_i}\in\E(\lp_i)\big\}
=\natexLP_i(f_i\vert B_i),
\end{equation*}
using Equation~\ref{eq:naturalextension} to establish the last equality.
\end{proof}
\vspace{-16pt}

\begin{proposition}\label{prop:productindependent:LP}
The restriction of $\lp_1\otimes\lp_2$ to $\C$ is an independent product of $\lp_1$ and $\lp_2$.
\end{proposition}
\begin{proof}{\bf of Proposition~\ref{prop:productindependent:LP}~}
Since we know from Proposition~\ref{prop:productcoherent:LP} that $\lp_1\otimes\lp_2$ is a coherent lower prevision on $\C(\states_1\times\states_2)$, it follows from Definition~\ref{def:cohlp} that its restriction to $\C$ is coherent as well. Due to Definition~\ref{def:independentproduct:LP}, it remains to show that this restriction of $\lp_1\otimes\lp_2$ to $\C$ is epistemically independent and that it coincides with $\lp_1$ and $\lp_2$ on their domain. Epistemic independence follows trivially from Definition~\ref{def:epistemicindependence:LP} and Proposition~\ref{prop:independent:localnatex}. Hence, it remains to prove that the restriction of $\lp_1\otimes\lp_2$ to $\C$ coincides with $\lp_1$ and $\lp_2$ on their domain, or equivalently, that
\begin{equation*}
(\lp_1\otimes\lp_2)(f_i\vert B_i)=\lp_i(f_i\vert B_i)
\text{~~for all $i\in\{1,2\}$ and $(f_i,B_i)\in\C_i$.}
\end{equation*}
So fix any $i\in\{1,2\}$ and $(f_i,B_i)\in\C_i$. We then find that indeed, as desired,
\begin{equation*}
(\lp_1\otimes\lp_2)(f_i\vert B_i)
=\natexLP_i(f_i\vert B_i)
=\lp_i(f_i\vert B_i),
\end{equation*}
where the first equality follows from Proposition~\ref{prop:independent:localnatex} and the second equality follows from Equation~\eqref{eq:naturalextension} and Proposition~\ref{prop:smallestSDGfromLP}.
\end{proof}

\begin{proof}{\bf of Theorem~\ref{theo:natext:LP}~}
Since we know from Proposition~\ref{prop:productindependent:LP} that the restriction of $\lp_1\otimes\lp_2$ to $\C$ is an independent product of $\lp_1$ and $\lp_2$, it suffices to prove that any other such independent product of $\lp_1$ and $\lp_2$ dominates $\lp_1\otimes\lp_2$ on $\C$.

So let $\lp$ be any independent product of $\lp_1$ and $\lp_2$. 
Definition~\ref{def:independentproduct:LP} then implies that $\lp$ is an epistemically independent coherent conditional lower prevision on $\C$ that coincides with $\lp_1$ and $\lp_2$ on their domain. Let $\A_{\lp}$ be the corresponding set of gambles, as defined by Equation~\eqref{eq:AfromLP}, and let \mbox{$\desir\coloneqq\E(\lp)=\E(\A_{\lp})$}. We then know from Proposition~\ref{prop:smallestSDGfromLP} that $\desir$ is a coherent set of desirable gambles on $\states_1\times\states_2$ and that $\lp_\desir$ coincides with $\lp$ on $\C$.
In the remainder of this proof, we will show that $\E(\lp_1)\otimes\E(\lp_2)\subseteq\desir$. Because of Equation~\eqref{eq:indnatex:LP}, this clearly implies that $\lp_\desir(f\vert B)\geq(\lp_1\otimes\lp_2)(f\vert B)$ for all $(f,B)\in\C$. Since $\lp_\desir$ coincides with $\lp$ on $\C$, this implies that $\lp$ dominates $\lp_1\otimes\lp_2$ on $\C$, thereby concluding the proof.

Let $\desir_1\coloneqq\E(\lp_1)$ and let $\A_{2\to1}$ be the corresponding set of gambles on $\states_1\times\states_2$, as defined by Equation~\eqref{eq:A21s}. We will now prove that $\A_{2\to1}\subseteq\desir$. So consider any $f_1\in\desir_1$ and any $B_2\in\B_2\cup\{\states_2\}$. We need to prove that $f_1(X_1)\ind{B_2}(X_2)\in\desir$. Since $f_1\in\desir_1=\E(\lp_1)=\posi(\A_{\lp_1}\cup\mathcal{G}_{>0}(\states_1))$, it follows from Equation~\eqref{eq:posi} that there are $n\in\nats$ and, for all $i\in\{1,\dots,n\}$, $\lambda_i\in\reals_{>0}$ and $g_i\in\A_{\lp_1}\cup\mathcal{G}_{>0}(\states_1)$ such that $f_1=\sum_{i=1}^n\lambda_ig_i$.

For any $i\in\{1,\dots,n\}$, we now let $h_i(X_1,X_2)\coloneqq g_i(X_1)\ind{B_2}(X_2)\in\mathcal{G}(\states_1\times\states_2)$. As we will show, this gamble $h_i$ belongs to $\desir$. We consider two cases: $g_i\in\mathcal{G}_{>0}(\states_1)$ and $g_i\notin\mathcal{G}_{>0}(\states_1)$. If $g_i\in\mathcal{G}_{>0}(\states_1)$, then $h_i\in\mathcal{G}_{>0}(\states_1\times\states_2)$, which, since $\desir$ is a coherent set of desirable gambles on $\states_1\times\states_2$, implies that $h_i\in\desir$. If $g_i\not\in\mathcal{G}_{>0}$, then since $g_i\in\A_{\lp_1}\cup\mathcal{G}_{>0}(\states_1)$, it follows that $g_i\in\A_{\lp_1}$, which implies that there are $(f'_1,B_1)\in\C_1$ and  $\mu<\lp_1(f'_1\vert B_1)$ such that $g_i=[f'_1-\mu]\ind{B_1}$. Furthermore, since $\lp$ coincides with $\lp_1$ on its domain, we also know that $\lp_1(f'_1\vert B_1)=\lp(f'_1\vert B_1)$. If $B_2=\states_2$, Equation~\eqref{eq:AfromLP} therefore implies that $h_i\in\A_{\lp}\subseteq\desir$ because $\ind{B_2}=1$. If $B_2\neq\states_2$, then $B_2\in\B_2$. Since $\lp$ is epistemically independent, this implies that $\lp(f'_1\vert B_1)=\lp(f'_1\vert B_1\cap B_2)$. Hence, here too, Equation~\eqref{eq:AfromLP} implies that $h_i\in\A_{\lp}\subseteq\desir$---because $\ind{B_1\cap B_2}=\ind{B_1}\ind{B_2}$.

In summary then, we have found that $h_i\in\desir$ for all $i\in\{1,\dots,n\}$. Since $f_1=\sum_{i=1}^n\lambda_ig_i$, this implies that
\vspace{-5pt}
\begin{equation*}
f_1(X_1)\ind{B_2}(X_2)
=
\left(\sum_{i=1}^n
\lambda_ig_i(X_1)
\right)
\ind{B_2}(X_2)
=
\sum_{i=1}^n
\lambda_i
g_i(X_1)
\ind{B_2}(X_2)
=
\sum_{i=1}^n
\lambda_i
h_i(X_1,X_2)
\in\desir,
\vspace{4pt}
\end{equation*}
where the inclusion holds because $\desir$ is coherent. Since this is true for every $f_1\in\desir_1$ and every $B_2\in\B_2\cup\{\states_2\}$, it follows that $\A_{2\to1}\subseteq\desir$. 
Using a completely analogous argument, it also follows that $\A_{1\to2}\subseteq\desir$, with $\A_{1\to2}$ defined by Equation~\eqref{eq:A12s} for $\desir_2\coloneqq\E(\lp_2)$.
Hence, we find that $\A_{1\to2}
\cup
\A_{2\to1}
\subseteq\desir$, and therefore,  that 
\begin{equation*}
\E(\lp_1)\otimes\E(\lp_2)
=
\desir_1\otimes\desir_2
=
\E
\left(
\A_{1\to2}
\cup
\A_{2\to1}
\right)
\subseteq\E(\desir)=\desir, 
\end{equation*}
where the second equality follows from Equation~\eqref{eq:indnatext:SDG}, the inclusion follows from Lemma~\ref{lemma:nestedpropsofposandE}, and the last equality follows from Lemma~\ref{lemma:natextDisD}.
\end{proof}
\vspace{-10pt}

\subsection{Proofs and Additional Material for Section~\ref{sec:choiceofevents}}
\vspace{5pt}

\begin{proof}{\bf of Proposition~\ref{prop:addfiniteunionstoB}~}
We only prove the result for $\desir_1\otimes\desir_2$. The result for $\lp_1\otimes\lp_2$ then follows trivially from Equation~\eqref{eq:indnatex:LP}.

Let $\desir_1\otimes\desir_2$ be the independent natural extension that corresponds to $\B_1$ and $\B_2$, as defined by Equations~\eqref{eq:indnatext:SDG}--\eqref{eq:A21s}, and let $\desir_1\otimes'\desir_2$ be the independent natural extension that corresponds to $\B'_1$ and $\B'_2$, defined by
\begin{equation*}%\label{eq:indnatext:SDG}
\desir_1\otimes'\desir_2
\coloneqq
\E
\left(
\A'_{1\to2}
\cup
\A'_{2\to1}
\right),
\vspace{-4pt}
\end{equation*}
with
\begin{equation*}%\label{eq:A12sprime}
\A'_{1\to2}
\coloneqq
\left\{
f_2(X_2)\ind{B'_1}(X_1)
\colon
f_2\in\desir_2, 
B'_1\in\B'_1\cup\{\states_1\}
\right\}
\vspace{-3pt}
\end{equation*}
and
\begin{equation*}%\label{eq:A21sprime}
\A'_{2\to1}
\coloneqq
\left\{
f_1(X_1)\ind{B'_2}(X_2)
\colon
f_1\in\desir_1, 
B'_2\in\B'_2\cup\{\states_2\}
\right\}.
\vspace{8pt}
\end{equation*}
Then as explained in the main text, in the paragraph that precedes Proposition~\ref{prop:addfiniteunionstoB}, we have that $\desir_1\otimes\desir_2\subseteq\desir_1\otimes'\desir_2$. It remains to prove that $\desir_1\otimes'\desir_2\subseteq\desir_1\otimes\desir_2$. 

Fix any $f_2\in\desir_2$ and $B'_1\in\B'_1\cup\{\states_1\}$. We will prove that $f_2(X_2)\ind{B'_1}(X_1)\in\desir_1\otimes\desir_2$. If $B'_1=\states_1$, this follows trivially from Equations~\eqref{eq:indnatext:SDG} and~\eqref{eq:A12s}. Otherwise, it follows from our assumptions that there is some $m\in\nats$ and, for all $k\in\{1,\dots,m\}$, some $B_{1,k}\in\B_1$ such that $B'_1$ is a finite disjoint union of the events $\{B_{1,k}\}_{1\leq k\leq m}$, which implies that $\ind{B'_1}=\sum_{k=1}^m\ind{B_{1,k}}$ and therefore also that $f_2(X_2)\ind{B'_1}(X_1)=\sum_{k=1}^mf_2(X_2)\ind{B_1}(X_1)$. Hence, Equations~\eqref{eq:indnatext:SDG} and~\eqref{eq:A12s} again imply that $f_2(X_2)\ind{B'_1}(X_1)\in\desir_1\otimes\desir_2$. Since this is true for every $f_2\in\desir_2$ and $B'_1\in\B'_1\cup\{\states_1\}$, it follows that $\A'_{1\to2}\subseteq\desir_1\otimes\desir_2$. Using a completely analogous argument, we also infer that $\A'_{2\to1}\subseteq\desir_1\otimes\desir_2$. The result now follows because $\A'_{1\to2}
\cup
\A'_{2\to1}\subseteq\desir_1\otimes\desir_2$ implies that
\begin{equation*}
\desir_1\otimes'\desir_2
=
\E
\left(
\A'_{1\to2}
\cup
\A'_{2\to1}
\right)
\subseteq
\E\left(
\desir_1\otimes\desir_2
\right)
=
\desir_1\otimes\desir_2,
\end{equation*}
using Lemma~\ref{lemma:nestedpropsofposandE} for the inclusion and Lemma~\ref{lemma:natextDisD} and Proposition~\ref{prop:productcoherent:SDG} for the last equality.
\end{proof}
\vspace{-10pt}

\subsection{Proofs and Additional Material for Section~\ref{sec:factadd}}
\vspace{5pt}

\begin{lemma}\label{lemma:fact-add-simple-geq}
For any $f\in\gambleson{1}$ and $h\in\gambleson{2}$ and any simple $\B_1$-measurable $g\in\mathcal{G}_{\geq0}(\states_1)$, we have that
\begin{equation*}
(\lp_1\otimes\lp_2)(f+gh)
\geq\natexLP_1\big(f+g\natexLP_2(h)\big).
\vspace{8pt}
\end{equation*}
\end{lemma}
\begin{proof}{\bf of Lemma~\ref{lemma:fact-add-simple-geq}~}
Since $g\in\mathcal{G}_{\geq0}(\states_1)$ is a simple $\B$-measurable gamble, we know from Definition~\ref{def:measurable:simple} that there are $c_0\in\reals_{\geq0}$, $n\in\natswith$ and, for all $i\in\{1,\dots,n\}$, $c_i\in\reals_{\geq0}$ and $B_i\in\B_1$, such that $g=c_0+\sum_{i=1}^nc_i\ind{B_i}$. Furthermore, since we know from Proposition~\ref{prop:productcoherent:LP} that $\lp_1\otimes\lp_2$ is coherent, it follows from Proposition~\ref{prop:propertiesofLP} that $\lp_1\otimes\lp_2$ satisfies~\ref{def:lowerprev:homo},~\ref{def:lowerprev:superadditive} and~\ref{def:lowerprev:GBR}. Finally, since $\natexLP_2$ is coherent, we know from Proposition~\ref{prop:propertiesofLP} that it satisfies~\ref{def:lowerprev:constantadditivity}. Therefore, we find that
\begin{align*}
(\lp_1\otimes\lp_2)&(f+gh)
=
(\lp_1\otimes\lp_2)\Big(
f+g\natexLP_2(h)+\big(c_0+\sum_{i=1}^nc_i\ind{B_i}\big)[h-\natexLP_2(h)]
\Big)\\
&\geq
(\lp_1\otimes\lp_2)\big(
f+g\natexLP_2(h)\big)
+
c_0
(\lp_1\otimes\lp_2)\big(h-\natexLP_2(h)
\big)
+
\sum_{i=1}^nc_i
(\lp_1\otimes\lp_2)\big(\ind{B_i}[h-\natexLP_2(h)]
\big)\\
&=
\natexLP_1\big(
f+g\natexLP_2(h)\big)
+
c_0
\natexLP_2\big(h-\natexLP_2(h)
\big)
+
\sum_{i=1}^nc_i
(\lp_1\otimes\lp_2)\big(\ind{B_i}[h-(\lp_1\otimes\lp_2)(h\vert B_i)]
\big)\\
&=
\natexLP_1\big(
f+g\natexLP_2(h)\big)
+
c_0
\big(\natexLP_2(h)-\natexLP_2(h)
\big)
=\natexLP_1\big(
f+g\natexLP_2(h)\big),\\[-8pt]
\end{align*}
where the first equality follows because $g=c_0+\sum_{i=1}^nc_i\ind{B_i}$, where the first inequality follows because $\lp_1\otimes\lp_2$ satisfies~\ref{def:lowerprev:superadditive} and~\ref{def:lowerprev:homo}, where the second equality follows from Proposition~\ref{prop:independent:localnatex}, and where the third equality follows because $\natexLP_2$ satisfies~\ref{def:lowerprev:constantadditivity} and $\lp_1\otimes\lp_2$ satisfies~\ref{def:lowerprev:GBR}.
\end{proof}
\vspace{-16pt}

\begin{lemma}\label{lemma:fact-add-simple-leq}
For any $f\in\gambleson{1}$ and $h\in\gambleson{2}$ and any simple $\B_1$-measurable $g\in\mathcal{G}_{\geq0}(\states_1)$, we have that
\begin{equation*}
(\lp_1\otimes\lp_2)(f+gh)
\leq\natexLP_1\big(f+g\natexLP_2(h)\big).\vspace{8pt}
\end{equation*}
\end{lemma}
\begin{proof}{\bf of Lemma~\ref{lemma:fact-add-simple-leq}~}
Since $\natexLP_2$ is a coherent conditional lower prevision on $\C(\states_2)$, we know from Proposition~\ref{prop:lowerenvelope} that there is a conditional linear prevision $P_2$ on $\C(\states_2)$ such that $P_2(h)=\natexLP_2(h)$ and $P_2\geq\natexLP_2$. Similarly, since $\natexLP_1$ is a coherent conditional lower prevision on $\C(\states_1)$, we know from Proposition~\ref{prop:lowerenvelope} that there is a conditional linear prevision $P_1$ on $\C(\states_1)$ such that $P_1\big(f+g\natexLP_2(h)\big)=\natexLP_1\big(f+g\natexLP_2(h)\big)$ and $P_1\geq\natexLP_1$.

Consider now any $i\in\{1,2\}$. We then know from Proposition~\ref{prop:naturalextension} that $\natexLP_i$ coincides with $\lp_i$ on $\C_i$. Therefore, and because $P_i\geq\natexLP_i$, we also know that $P_i$ dominates $\lp_i$ on $\C_i$. Due to Equation~\eqref{eq:AfromLP}, this implies that $\A_{\lp_i}\subseteq\A_{P_i}$ and therefore, using Lemma~\ref{lemma:nestedpropsofposandE}, also that $\E(\lp_i)\subseteq\E(P_i)$. 
Since this is true for every $i\in\{1,2\}$, it follows from Equation~\eqref{eq:indnatext:SDG} and Lemma~\ref{lemma:nestedpropsofposandE} that $\E(\lp_1)\otimes\E(\lp_2)\subseteq\E(P_1)\otimes\E(P_2)$, and therefore, because of Equation~\eqref{eq:indnatex:LP}, that $\lp_1\otimes\lp_2\leq P_1\otimes P_2$. 

The result can now be proved as follows. First, since $\lp_1\otimes\lp_2\leq P_1\otimes P_2$, we find that 
\begin{equation}\label{eq:lemma:fact-add-simple-leq:1}
(\lp_1\otimes\lp_2)(f+gh)
\leq
(P_1\otimes P_2)(f+gh).
\end{equation}
Secondly, since we know from Proposition~\ref{prop:productcoherent:LP} that $(P_1\otimes P_2)$ is coherent, it follows from Proposition~\ref{prop:propertiesofLP} that $(P_1\otimes P_2)$ satisfies~\ref{def:lowerprev:lowerbelowupper}, which implies that
\begin{equation}\label{eq:lemma:fact-add-simple-leq:2}
(P_1\otimes P_2)(f+gh)
\leq
-(P_1\otimes P_2)(-f-gh)
\leq
-P_1\big(-f+gP_2(-h)\big),
\end{equation}
using Lemma~\ref{lemma:fact-add-simple-geq} for the second inequality.
Finally, we also know that
\begin{equation}\label{eq:lemma:fact-add-simple-leq:3}
%(\lp_1\otimes\lp_2)(f+gh)
%\leq
-P_1\big(-f+gP_2(-h)\big)
%&=
%-P_1\big(-f-gP_2(h)\big)\\
=
P_1\big(f+gP_2(h)\big)
%=
%P_1\big(f+g\natexLP_2(h)\big)
=
\natexLP_1\big(f+g\natexLP_2(h)\big),
\end{equation}
where the first equality follows from Definitions~\ref{def:prev} and~\ref{def:linearprev} because $P_1$ and $P_2$ are conditional linear previsions, and where the second equality follows because $P_2(h)=\natexLP_2(h)$ and $P_1\big(f+g\natexLP_2(h)\big)=\natexLP_1\big(f+g\natexLP_2(h)\big)$.
By combining Equations~\eqref{eq:lemma:fact-add-simple-leq:1}--\eqref{eq:lemma:fact-add-simple-leq:3}, the result is now immediate.
\end{proof}
\vspace{-16pt}

\begin{proposition}\label{prop:fact-add-simple}
For any $f\in\gambleson{1}$ and $h\in\gambleson{2}$ and any simple $\B_1$-measurable $g\in\mathcal{G}_{\geq0}(\states_1)$, we have that
\begin{equation*}
(\lp_1\otimes\lp_2)(f+gh)
=\natexLP_1\big(f+g\natexLP_2(h)\big).
\vspace{5pt}
\end{equation*}
\end{proposition}
\begin{proof}{\bf of Proposition~\ref{prop:fact-add-simple}~}
Immediate consequence of Lemmas~\ref{lemma:fact-add-simple-geq} and~\ref{lemma:fact-add-simple-leq}.
\end{proof}
\vspace{-7pt}

\begin{proof}{\bf of Theorem~\ref{theo:fact-add-measurable}~}
Since $g\in\mathcal{G}_{\geq0}(\states_1)$ is $\B_1$-measurable, we know from Definition~\ref{def:measurable:uniform} that there is a sequence $\{g_n\}_{n\in\nats}$ of simple $\B_1$-measurable gambles in $\mathcal{G}_{\geq0}(\states_1)$ such that $g_n$ converges uniformly to $g$.
This also implies that $f+g_n\natexLP_2(h)$ converges uniformly to $f+g\natexLP_2(h)$ and, since $h$ is a gamble and therefore by definition bounded, that $f+g_nh$ converges uniformly to $f+gh$. The result now follows from the following series of equalities:
\begin{equation*}
(\lp_1\otimes\lp_2)(f+gh)
=
\lim_{n\to+\infty}
(\lp_1\otimes\lp_2)(f+g_nh)
=
\lim_{n\to+\infty}
\natexLP_1\big(f+g_n\natexLP_2(h)\big)
=
\natexLP_1\big(f+g\natexLP_2(h)\big).
\end{equation*}
The first of these equalities holds because it follows from Propositions~\ref{prop:productcoherent:LP} and~\ref{prop:propertiesofLP} that $\lp_1\otimes\lp_2$ satisfies~\ref{def:lowerprev:uniformcontinuity}. The second equality follows from Proposition~\ref{prop:fact-add-simple}. The third equality holds because the coherence of $\natexLP_1$ allows us to infer from Proposition~\ref{prop:propertiesofLP} that $\natexLP_1$ satisfies~\ref{def:lowerprev:uniformcontinuity}.
\end{proof}

\begin{proof}{\bf of Corollary~\ref{corol:fact-measurable}~}
Let $f\coloneqq0\in\mathcal{G}(\states_i)$. We then know from Theorem~\ref{theo:fact-add-measurable} that 
\begin{equation*}
(\lp_1\otimes\lp_2)(gh)
=(\lp_1\otimes\lp_2)(f+gh)
=\natexLP_i\big(f+g\natexLP_j(h)\big)
=\natexLP_i\big(g\natexLP_j(h)\big).
\end{equation*}
The result can now be inferred from the non-negative homogeneity---\ref{def:lowerprev:homo}---of $\natexLP_i$ that is implied by its coherence. If $\natexLP_j(h)\geq0$, we simply apply the non-negative homogeneity for $\lambda\coloneqq\natexLP_j(h)$. If $\natexLP_j(h)\leq0$, we apply it for $\lambda\coloneqq-\natexLP_j(h)$ and combine this with the fact that $\natexUP_i(g)\coloneqq-\natexLP_i(-g)$.
\end{proof}

\begin{proof}{\bf of Corollary~\ref{corol:add}~}
Let $g\coloneqq1$. Then $g$ belongs to $\mathcal{G}_{\geq0}(\states_1)$ and is $\B_1$-measurable. Therefore, we know from Theorem~\ref{theo:fact-add-measurable} that
\begin{equation*}
(\lp_1\otimes\lp_2)(f+h)
=(\lp_1\otimes\lp_2)(f+gh)
=\natexLP_1\big(f+g\natexLP_2(h)\big)
=\natexLP_1\big(f+\natexLP_2(h)\big).
\end{equation*}
The result now follows from the constant additivity---\ref{def:lowerprev:constantadditivity}---of $\natexLP_1$ that is implied by its coherence.
\end{proof}


\begin{thebibliography}{17}
\providecommand{\natexlab}[1]{#1}
\providecommand{\url}[1]{\texttt{#1}}
\expandafter\ifx\csname urlstyle\endcsname\relax
  \providecommand{\doi}[1]{doi: #1}\else
  \providecommand{\doi}{doi: \begingroup \urlstyle{rm}\Url}\fi

\bibitem[Couso et~al.(1999)Couso, Moral, and Walley]{Couso:1999wh}
I.~Couso, S.~Moral, and P.~Walley.
\newblock {Examples of Independence for Imprecise Probabilities}.
\newblock In \emph{ISIPTA '99: Proceedings of the First International Symposium
  on Imprecise Probabilities and their Applications}, pages 121--130. 1999.

\bibitem[Cozman(2013)]{Cozman:2013us}
F.~G. Cozman.
\newblock {Independence for sets of full conditional probabilities, sets of
  lexicographic probabilities, and sets of desirable gambles}.
\newblock In \emph{ISIPTA '13: Proceedings of the Eighth International
  Symposium on Imprecise Probability: Theory and Applications}, pages 87--97.
  2013.

\bibitem[De~Bock(2015)]{de2015credal}
J.~De~Bock.
\newblock \emph{Credal networks under epistemic irrelevance: theory and
  algorithms}.
\newblock PhD thesis, Ghent University, 2015.

% \bibitem[De~Bock(2017)]{debock2017williamsArXiv}
% J.~De~Bock.
% \newblock \emph{Independent Natural Extension for Infinite Spaces: Williams-Coherence to the Rescue}.
% \newblock ArXiv report 1701.07295, 2017.

\bibitem[De~Bock and de~Cooman(2014)]{DeBock:2014ts}
J.~De~Bock and G.~de~Cooman.
\newblock {An efficient algorithm for estimating state sequences in imprecise
  hidden Markov models}.
\newblock \emph{Journal of Artificial Intelligence Research}, 50:\penalty0
  189--233, 2014.

\bibitem[de~Cooman and Miranda(2012)]{deCooman:2012vba}
G.~de~Cooman and E.~Miranda.
\newblock {Irrelevant and independent natural extension for sets of desirable
  gambles}.
\newblock \emph{Journal of Artificial Intelligence Research}, 45:\penalty0
  601--640, 2012.

\bibitem[de~Cooman et~al.(2010)de~Cooman, Hermans, Antonucci, and
  Zaffalon]{deCooman:2010gd}
G.~de~Cooman, F.~Hermans, A.~Antonucci, and M.~Zaffalon.
\newblock {Epistemic irrelevance in credal nets: the case of imprecise Markov
  trees}.
\newblock \emph{International Journal of Approximate Reasoning}, 51\penalty0
  (9):\penalty0 1029--1052, 2010.

\bibitem[de~Cooman et~al.(2011)de~Cooman, Miranda, and
  Zaffalon]{deCooman:2011ey}
G.~de~Cooman, E.~Miranda, and M.~Zaffalon.
\newblock {Independent natural extension}.
\newblock \emph{Artificial Intelligence}, 175\penalty0 (12):\penalty0
  1911--1950, 2011.

\bibitem[Miranda and Zaffalon(2015)]{Miranda2015460}
E.~Miranda and M.~Zaffalon.
\newblock Independent products in infinite spaces.
\newblock \emph{Journal of Mathematical Analysis and Applications},
  425\penalty0 (1):\penalty0 460 -- 488, 2015.

\bibitem[Nielsen(1997)]{Nielsen1997}
O.~A. Nielsen.
\newblock \emph{{An introduction to integration and measure theory}}.
\newblock Wiley, 1997.

\bibitem[Pelessoni and Vicig(2009)]{Pelessoni:2009co}
R.~Pelessoni and P.~Vicig.
\newblock {Williams coherence and beyond}.
\newblock \emph{International Journal of Approximate Reasoning}, 50\penalty0
  (4):\penalty0 612--626, 2009.

\bibitem[Quaeghebeur(2014)]{Quaeghebeur:2014tjb}
E.~Quaeghebeur.
\newblock {Desirability}.
\newblock In T.~Augustin, F.~P.~A. Coolen, G.~de~Cooman, and M.~C.~M. Troffaes,
  editors, \emph{Introduction to Imprecise Probabilities}, pages 1--27. John
  Wiley {\&} Sons, 2014.

\bibitem[Vicig(2000)]{Vicig:2000vh}
P.~Vicig.
\newblock {Epistemic independence for imprecise probabilities}.
\newblock \emph{International Journal of Approximate Reasoning}, 24\penalty0
  (2-3):\penalty0 235--250, 2000.

\bibitem[Walley(1991)]{Walley:1991vk}
P.~Walley.
\newblock \emph{{Statistical reasoning with imprecise probabilities}}.
\newblock Chapman and Hall, London, 1991.

\bibitem[Walley(2000)]{Walley:2000ef}
P.~Walley.
\newblock {Towards a unified theory of imprecise probability}.
\newblock \emph{International Journal of Approximate Reasoning}, 24\penalty0
  (2-3):\penalty0 125--148, 2000.

\bibitem[Williams(1975)]{williams1975}
P.~M. Williams.
\newblock {Notes on conditional previsions}.
\newblock Technical report, School of Mathematical and Physical Science,
  University of Sussex, 1975.

\bibitem[Williams(2007)]{Williams:2007eu}
P.~M. Williams.
\newblock {Notes on conditional previsions}.
\newblock \emph{International Journal of Approximate Reasoning}, 44\penalty0
  (3):\penalty0 366--383, 2007.

\end{thebibliography}
\end{document}